\documentclass[a4paper]{amsart}

\usepackage[applemac]{inputenc}
\usepackage[T1]{fontenc}
\usepackage[english]{babel}
\usepackage{babelbib}
\usepackage{enumerate}
\usepackage{url}
\usepackage{amssymb}
\usepackage[cal=boondoxo]{mathalfa}
\usepackage{scalerel}
\usepackage{aliascnt}

\DeclareFontFamily{U}{skulls}{}
\DeclareFontShape{U}{skulls}{m}{n}{ <-> skull }{}

\usepackage[hypertexnames=false,
 	pdfpagemode=UseNone,
 	breaklinks=true,
 	extension=pdf,
 	colorlinks=true,
 	linkcolor=blue,
 	citecolor=red,
 	urlcolor=blue,
 ]{hyperref} 
\usepackage{color}

\usepackage{tikz}
\usepackage{tikz-cd}

\usetikzlibrary{matrix}
\usetikzlibrary{arrows}

\usepackage{graphicx}

\usepackage{cleveref}

\mathcode`A="7041 \mathcode`B="7042 \mathcode`C="7043 \mathcode`D="7044
\mathcode`E="7045 \mathcode`F="7046 \mathcode`G="7047 \mathcode`H="7048
\mathcode`I="7049 \mathcode`J="704A \mathcode`K="704B \mathcode`L="704C
\mathcode`M="704D \mathcode`N="704E \mathcode`O="704F \mathcode`P="7050
\mathcode`Q="7051 \mathcode`R="7052 \mathcode`S="7053 \mathcode`T="7054
\mathcode`U="7055 \mathcode`V="7056 \mathcode`W="7057 \mathcode`X="7058
\mathcode`Y="7059 \mathcode`Z="705A

\newcommand{\bbC}{\mathbb{C}}

\newcommand{\bbF}{\mathbb{F}}

\newcommand{\bbN}{\mathbb{N}}

\newcommand{\bbP}{\mathbb{P}}
\newcommand{\bbQ}{\mathbb{Q}}

\newcommand{\bbS}{\mathbb{S}}

\newcommand{\bbV}{\mathbb{V}}

\newcommand{\bbZ}{\mathbb{Z}}

\newcommand{\Gm}{\mathbb{G}_m}

\newcommand{\cA}{\mathcal{A}}

\newcommand{\cC}{\mathcal{C}}

\newcommand{\cE}{\mathcal{E}}

\newcommand{\cG}{\mathcal{G}}
\newcommand{\cH}{\mathcal{H}}

\newcommand{\cK}{\mathcal{K}}
\newcommand{\cL}{\mathcal{L}}

\newcommand{\cN}{\mathcal{N}}
\newcommand{\cO}{\mathcal{O}}

\newcommand{\cS}{\mathcal{S}}
\newcommand{\cT}{\mathcal{T}}

\newcommand{\cV}{\mathcal{V}}
\newcommand{\cW}{\mathcal{W}}
\newcommand{\cX}{\mathcal{X}}
\newcommand{\cY}{\mathcal{Y}}
\newcommand{\cZ}{\mathcal{Z}}

\newcommand{\rB}{\textup{B}}

\newcommand{\rD}{\textup{D}}

\newcommand{\rH}{\textup{H}}

\newcommand{\rR}{\textup{R}}

\newcommand{\rT}{\textup{T}}

\newcommand{\frS}{\mathfrak{S}}

\newcommand{\an}{\textup{an}}

\renewcommand{\ss}{\textup{ss}}

\newcommand{\sym}{{\bbS}}

\newcommand{\et}{\textup{\'et}}
\newcommand{\dR}{\textup{dR}}

\newcommand{\cris}{\textup{crys}}

\renewcommand{\top}{\textup{top}}

\newcommand{\into}{\hookrightarrow}
\newcommand{\too}{\longrightarrow}
\renewcommand{\phi}{\varphi}
\renewcommand{\epsilon}{\varepsilon}
\renewcommand{\ker}{\Ker}

\newcommand{\iso}{\simeq}

\newcommand{\Rep}{\textup{Rep}}
\newcommand{\Dbc}{{\mathrm{D}^b_c}}

\newcommand{\Perv}{\textup{Perv}}

\newcommand{\PLambda}{\scalerel*{{\rotatebox[origin=c]{180}{$\bbV$}}}{\bbV}}

\newcommand{\intoo}{\lhook\joinrel\longrightarrow}

\DeclareMathOperator{\Vect}{Vect}
\DeclareMathOperator{\Alb}{Alb}

\DeclareMathOperator{\Gal}{Gal}

\DeclareMathOperator{\rk}{rk}
\DeclareMathOperator{\pr}{pr}

\DeclareMathOperator{\Spec}{Spec}

\DeclareMathOperator{\cEnd}{\mathcal{E}\hspace{-2pt}\textit{nd}}

\DeclareMathOperator{\im}{Im}
\DeclareMathOperator{\Ker}{Ker}

\DeclareMathOperator{\Flag}{Flag}

\newcommand{\red}{{\mathrm{red}}}

\DeclareMathOperator{\Gr}{Gr}

\DeclareMathOperator{\Pic}{Pic}
\DeclareMathOperator{\Stab}{Stab}

\DeclareMathOperator{\Sym}{Sym}

\DeclareMathOperator{\Aut}{Aut}
\DeclareMathOperator{\GL}{GL}
\DeclareMathOperator{\SL}{SL}
\DeclareMathOperator{\SO}{SO}
\DeclareMathOperator{\Sp}{Sp}
\DeclareMathOperator{\GO}{GO}

\DeclareMathOperator{\GSp}{GSp}
\DeclareMathOperator{\Spin}{Spin}

\DeclareMathOperator{\Supp}{Supp}
\DeclareMathOperator{\Lie}{Lie}

\DeclareMathOperator{\gr}{gr}
\DeclareMathOperator{\id}{id}

\DeclareMathOperator{\alb}{alb}

\DeclareMathOperator{\Characteristic}{Char}

\DeclareMathOperator{\rad}{rad}

\DeclareMathOperator{\codim}{codim}

\DeclareMathOperator{\Alt}{Alt}

\DeclareMathOperator{\cc}{cc}
\DeclareMathOperator{\CC}{CC}
\DeclareMathOperator{\Char}{Char}

\renewcommand{\le}{\leqslant}
\renewcommand{\ge}{\geqslant}

\theoremstyle{plain}
\newtheorem{theoremintro}{Theorem}

\newtheorem*{maintheorem-monodromy}{Main theorem (monodromy version)}
\newtheorem*{maintheorem-tannaka}{Main theorem (Tannaka version)}
\newtheorem*{corollaryintro}{Corollary}

\newtheorem{theorem}{Theorem}[section]

\newaliascnt{lemma}{theorem}
\newtheorem{lemma}[lemma]{Lemma}
\aliascntresetthe{lemma}
\crefname{lemma}{lemma}{lemmas}

\newaliascnt{proposition}{theorem}
\newtheorem{proposition}[proposition]{Proposition}
\aliascntresetthe{proposition}
\crefname{proposition}{proposition}{propositions}

\newaliascnt{corollary}{theorem}
\newtheorem{corollary}[corollary]{Corollary}
\aliascntresetthe{corollary}
\crefname{corollary}{corollary}{corollaries}

\newaliascnt{claim}{theorem}

\aliascntresetthe{claim}
\crefname{claim}{claim}{claims}

\newaliascnt{fact}{theorem}

\aliascntresetthe{fact}
\crefname{fact}{fact}{facts}

\newtheorem*{bigmonodromyintro}{Big Monodromy Criterion}
\newtheorem*{bigtannakaintro}{Big Tannaka group Criterion}

\newaliascnt{assumption}{equation}

\aliascntresetthe{assumption}
\crefname{assumption}{assumption}{assumptions}

\newaliascnt{condition}{equation}
\newtheorem{condition}[condition]{Condition}
\aliascntresetthe{condition}
\crefname{condition}{condition}{conditions}

\newtheorem*{ShafarevichConjecture}{Shafarevich conjecture for canonically polarized varieties}

\theoremstyle{definition}
\newtheorem*{definitionintro}{Definition}

\newaliascnt{definition}{theorem}
\newtheorem{definition}[definition]{Definition}
\aliascntresetthe{definition}
\crefname{definition}{definition}{definitions}

\newaliascnt{example}{theorem}
\newtheorem{example}[example]{Example}
\aliascntresetthe{example}
\crefname{example}{example}{examples}

\newaliascnt{notation}{theorem}

\aliascntresetthe{notation}
\crefname{notation}{notation}{notations}

\newaliascnt{remark}{theorem}

\aliascntresetthe{remark}
\crefname{remark}{remark}{remarks}

\newtheorem*{exampleintro}{Example}

\numberwithin{equation}{section}

\begin{document}

\title[The Shafarevich conjecture for varieties with spanned cotangent]{The Shafarevich conjecture for varieties with globally generated cotangent}

\begin{abstract}
We prove the Shafarevich conjecture for varieties with globally generated cotangent bundle, subject to mild numerical conditions.
\end{abstract}

\author[T. Kr\"amer]{Thomas Kr\"amer}
\address{Thomas Kr\"amer \\
	Fakult\"at f\"{u}r Mathematik\\
	Technische Universit\"at Chemnitz\\
	Reichenhainer Str.~39, 09126 Chemnitz\\
	Germany.}
\email{thomas.kraemer@mathematik.tu-chemnitz.de}

\author[M. Maculan]{Marco Maculan}

\address{ Marco Maculan \\
	Institut de Math\'ematiques de Jussieu \\
Sorbonne Universit\'e \\
4, place Jussieu \\
75005 Paris \\
France}
\email{marco.maculan@imj-prg.fr}




\maketitle

\thispagestyle{empty}

\setcounter{tocdepth}{1}

\tableofcontents

\section{Introduction}

Let $K$ be a number field and $\Sigma$ a finite set of primes in $K$. In this paper we prove new cases of the following conjecture:

\begin{ShafarevichConjecture} Fix $P \in \bbQ[t]$. Up to $K$-isomorphism there are only finitely many smooth projective canonically polarized varieties over $K$ with Hilbert polynomial $P$ and good reduction outside $\Sigma$.
\end{ShafarevichConjecture}

Here by the Hilbert polynomial of a variety we mean the one of its canonical bundle. We say that a smooth projective canonically polarized variety $Y$ over $K$ has \emph{good reduction outside $\Sigma$} if 
it admits a smooth projective model $\cY$ over the ring of $\Sigma$-integers $R=\cO_{K,\Sigma}$ such that the canonical bundle $\det \Omega^{1}_{\cY / R}$ is relatively ample. If such a 
model exists, then it is unique up to isomorphism. In what follows we will simply talk about \emph{good reduction}, assuming that $\Sigma$ has been fixed.\medskip

The above conjecture is a special case of the Lang-Vojta conjecture, predicting the nondensity of $\Sigma$-integral points on varieties of log general type over $K$. Indeed, the moduli spaces (or rather stacks) of canonically polarized varieties are known to be of log general type, as are all their subvarieties \cite{CampanaPaun}. The complex analogue of the above conjecture says that the moduli stacks of canonically polarized varieties do not contain entire curves, which is a celebrated result of Viehweg and Zuo \cite{ViehwegZuo}; however, these hyperbolicity properties will play no role in our arguments.\medskip

In the case of curves the conjecture is due to Shafarevich and was proven by Faltings on his path to the Mordell conjecture \cite{FaltingsMordell}. Since then, many cases of the conjecture and analogous finiteness results for unpolarized varieties have been proven \cite{SchollDelPezzo, AndreK3, She, TakamatsuK3, JavLoughFlag, JavanpeykarLoughran, JavanpeykarLondon, JavLoughPisa, JavLoughLondon, TakamatsuEnriques}.~All of  these rely on some classification of the varieties under consideration; this is not the case in~\cite{KM} where we proved the Shafarevich conjecture for subvarieties of abelian varieties with ample normal bundle, using the method of~\cite{LV, LS20} and the big monodromy criterion from~\cite{JKLM}. 
In this paper we extend the scope of the results in \cite{KM} to a much larger class of canonically polarized varieties:

\begin{definitionintro} A projective variety $Y$ over $k$ has \emph{amply generated cotangent bundle} if it is smooth, geometrically connected and the morphism $\rH^0(Y, \Omega^1_Y) \otimes_k \cO_Y \to \Omega^1_Y$ is surjective with nonzero and antiample (that is, with ample dual) kernel. 
\end{definitionintro}

If $Y$ has a $k$-rational point $y\in Y(k)$, then $Y$ has amply generated cotangent bundle if and only if the Albanese morphism $\alb_y\colon Y \to \Alb(Y)$ is unramified with ample normal bundle. In~\cite{KM} we discussed smooth projective varieties that embed in their Albanese variety with ample normal bundle, but the current class is much larger and goes beyond subvarieties of abelian varieties:

\begin{exampleintro} Let $k$ be a field of characteristic zero, and let $S$ be an irreducible smooth projective variety over $k$ such that the evaluation morphism $\rH^0(S, \Omega^1_S) \otimes_k \cO_S \to \Omega^1_S$ has locally free cokernel $\cC$ of rank $r$ (for instance take $h^0(S, \Omega^1_S)=0$). Let $A$ be an abelian variety over $k$, and let
\[ Y \; \subset \; A \times S \]
be a smooth complete intersection of ample divisors of dimension $d \ge 2$. Then the morphism $\Alb(Y) \to \Alb(S) \times A$ is an isomorphism by the Lefschetz theorem. If moreover $d + r < \dim A \times S$ and $Y$ is generic, the cotangent bundle of $Y$ is globally generated. The kernel $\cK$ of the evaluation morphism $\rH^0(Y, \Omega^1_Y) \otimes_k \cO_Y \to \Omega^1_Y$ is nonzero and sits in the short exact sequence
\[ 0 \too q ^\ast \cC^\vee \too \cN \too \cK^\vee \too 0 \]
where $\cN$ is the normal bundle of $Y$ in $A \times S$ and $q \colon Y \to S$ the morphism induced by the second projection. The vector bundle $\cN$ is ample, thus so is $\cK^\vee$ and with the above terminology $Y$ has amply generated cotangent bundle.\medskip

It may happen that $Y$ embeds into its Albanese over an algebraic closure of $k$, but the following construction gives new examples where this is not the case. Suppose that $S$ admits a connected finite \'etale cover $f \colon S' \to S$ of degree $> 1$ inducing an isomorphism
\[ \Alb(S') \; \stackrel{\sim}{\too} \; \Alb(S).\]
Over $k=\bbC$ such a cover exists whenever there is nontrivial torsion in $\rH_1(S, \bbZ)$, e.g.~for an Enriques surface $S$. In general, the preimage $Y' = S'\times_S Y$ has amply generated cotangent bundle and its Albanese variety is isomorphic to $\Alb(Y)$. So any Albanese morphism $Y' \to \Alb(Y')$ factors through~$Y$ and hence cannot be an embedding, and the same holds after any finite extension of the base field. This construction can be iterated by using $Y$ and $Y'$ in place of $S$ and~$S'$.
\end{exampleintro}

\subsection{Intrinsic results}

Varieties with amply generated cotangent bundle are in particular canonically polarized, so when they are defined over the number field $K$, our notion of good reduction from above applies. In the case of surfaces our main result takes the following form:

\begin{theoremintro}  \label{IntroThmIntrinsicSurfaces}
Fix an integer $c \ge 1$. Then up to $K$-isomorphism there are only finitely many projective surfaces $Y$ over $K$ with amply generated cotangent bundle, good reduction, $c_2(Y) = c$ and $h^0(Y, \Omega^1_Y) \ge 6$.
\end{theoremintro}

In fact we obtain an analogous result for varieties of arbitrary dimension, subject to certain mild numerical conditions:

\begin{theoremintro} \label{IntroThmIntrinsicGeneral}
Fix $P \in \bbQ[t]$ of degree $d$. Then up to $K$-isomorphism  there are only finitely many projective varieties $Y$ over $K$ with amply generated cotangent bundle, good reduction, Hilbert polynomial $P$, $h^0(Y, \Omega^1_Y) \ge 2 d + 2$ and satisfying the numerical conditions \ref{Eq:FirstNumericalCondition} and~\ref{Eq:NumericalConditions} below with $\pi = \id$. 
\end{theoremintro}

Conditions \ref{Eq:FirstNumericalCondition} and~\ref{Eq:NumericalConditions} will be discussed in the next section. They in particular hold if $d=2$, so~\cref{IntroThmIntrinsicGeneral} implies~\cref{IntroThmIntrinsicSurfaces} since by the Bogomolov-Miyaoka-Yau inequality, the Hilbert polynomial of a surface $S$ with amply generated cotangent bundle is controlled by $c_2(S)$ \cite[2.7]{KM}. For odd~$d$ and $h^0(Y, \Omega^1_Y) \ge 4d + 2$ the conditions \ref{Eq:FirstNumericalCondition} and~\ref{Eq:NumericalConditions} are empty, so~\cref{IntroThmIntrinsicGeneral} in particular implies:

\begin{corollaryintro} Fix $P \in \bbQ[t]$ of odd degree $d$. Then up to $K$-isomorphism  there are only finitely many projective varieties with amply generated cotangent bundle, good reduction, Hilbert polynomial~$P$ and $h^0(Y, \Omega^1_Y) \ge 4 d + 2$.
\end{corollaryintro}

\subsection{Numerical conditions} \label{sec:NumericalConditions} To formulate the two numerical conditions that go into~\cref{IntroThmIntrinsicGeneral} and into~\cref{IntroThmExtrinsic} below, let $Y$ be a smooth projective variety over a field $k$ of characteristic zero. The first condition compares the topological Euler characteristic $\chi_\top(Z)$ of the smooth projective variety $Z=Y\times Y$ to the Euler characteristic of the vector bundle $\Omega^d_Z$:

\begin{condition} \label{Eq:FirstNumericalCondition}
If $d=\dim Y\ge 4$ is even, assume that 
\[ (-1)^d \,\chi(\Omega^d_{Y \times Y}) \;\le\; \tfrac{1}{2} \,\chi_{\top}(Y \times Y). \]
\end{condition}

For the second numerical condition, we place ourselves in a slightly more general framework that will be used in~\cref{IntroThmExtrinsic}, starting from an arbitrary surjective morphism $\pi \colon \Alb(Y) \to A$ of abelian varieties; in fact in~\cref{IntroThmExtrinsic} the variety $Y$ will be given together with a morphism $Y \to A$ and we then tacitly assume $\pi$ to be induced by this morphism. Consider the vector subspace
\[ V\; := \; (\Lie A)^\vee \; \intoo \; (\Lie \Alb(Y))^\vee \; = \; \rH^0(Y, \Omega^1_Y) \]
and suppose that the induced morphism $\phi \colon V \otimes_k \cO_Y \to \Omega^1_Y$ is surjective. If $k$ is algebraically closed, let $\alb_y \colon Y \to \Alb(Y)$ be the Albanese morphism associated with $y \in Y(k)$. The composite morphism
\[ f \;=\; \pi\circ \alb_y \colon \quad Y \;  \too \; A\]
is unramified because $\phi$ is assumed to be surjective. Denote the \emph{stabilizer} of the image of $f$ by 
\[ G \; := \; \Stab_A(f(Y))\;=\;\{ a \in A(k) \colon f(Y) + a = f(Y)\}.\] 
Then $Y$ is of general type if and only if the group $G$ is finite; see \cref{sec:BasicProperties}. In this case we set
\[ 
\chi_{\top, \pi}(Y) := \frac{\chi_{\top}(Y)}{n} \quad \textup{for $n = |G|\cdot \deg (f)$,}
\]
where $\chi_{\top}(Y)$ is the topological Euler characteristic of $Y$. 
Note that $\chi_{\top, \pi}(Y)$ is an integer since it is equal to the topological Euler characteristic of the normalization of the quotient $f(Y)/G$ (this normalization is smooth by \cref{Lem:GeneralTypeCanonicallyPolarized}). We also say that $Y$ is \emph{symmetric with respect to $\pi$} if $f(Y) = a - f(Y)$ for some $a \in A(k)$. 
This notion and the definition of $\chi_{\top, \pi}(Y)$ do not depend on $y$, so by extending scalars they make sense also when the base field $k$ is not algebraically closed. Our second  condition excludes a few values of $\chi_{\top, \pi}(Y)$ in terms of $g=\dim A$:

\begin{condition} \label{Eq:NumericalConditions}
If $d=\dim Y\ge (g-1)/4$ and~$Y$ is symmetric with respect to $\pi$, assume that
\[
|\chi_{\top, \pi}(Y)| \;\neq\;
2^{2m - 1} 
\quad \text{for all $m\in \{3,\dots, d\}$ with $m\equiv d$ modulo $2$}.
\]
\end{condition}

\subsection{Extrinsic result} \label{sec:ExtrinsicResults} We will deduce \cref{IntroThmIntrinsicGeneral} from a general result for smooth projective varieties with an unramified morphism to an abelian variety. \medskip

To state it, as in \cite[1.4]{KM} we say that a geometrically integral variety $X$ over a field $k$ is a \emph{product}  resp. a \emph{symmetric power of a curve} if, after extending scalars to an algebraic closure of $k$, its normalization becomes isomorphic to a product of varieties of dimension $> 0$ resp.  a symmetric power of a curve. When $k$ is algebraically closed, a subvariety $X \subset A$ of dimension $d \ge 2$ is
\begin{itemize}
\item a product if and only if there are $X_1, X_2 \subset A$ of dimension $> 0$ with $X = X_1 + X_2$ such that the sum morphism induces an isomorphism  on normalizations 
\[\tilde{X}_1 \times \tilde{X}_2 \;\stackrel{\sim}{\too}\; \tilde{X}, \]
\item a symmetric power of a curve if and only if there is a curve $C \subset A$ with $X = C + \cdots + C$ such that the sum morphism induces on normalizations an isomorphism 
\[ \Sym^d \tilde{C} \;\stackrel{\sim}{\too}\; \tilde{X}. \]
\end{itemize}
We say that $X \subset A$ is \emph{isogenous to a symmetric power of a curve} resp. \emph{isogenous to a product} if $X / \Stab_A(X)$ is a symmetric power of a curve resp. a product. \medskip

Let $A$ be an abelian variety of dimension~$g$ over $K$ and $L$ an ample line bundle on it. Given smooth projective varieties $Y$ and $Y'$ with morphisms $f \colon Y \to A$ and $f' \colon Y' \to A$, we say that they are \emph{equivalent up to translation by points in $A(K)$} if there are an isomorphism $\phi \colon Y \to Y'$ and a point $a \in A(K)$ making the following diagram commutative:
\[\begin{tikzcd}[column sep=40pt]
Y \ar[r, "\phi"] \ar[d, swap, "f"] & Y' \ar[d, "f'"] \\
A \ar[r, "x\mapsto x+a"] & A
\end{tikzcd}
\]
With these notions we show:

\begin{theoremintro} \label{IntroThmExtrinsic}
Fix $P \in \bbQ[t]$ of degree $d < (g-1)/2$. Up to translation by points in $A(K)$, there are only finitely many canonically polarized varieties $Y$ over $K$ with good reduction together with an unramified morphism $f \colon Y \to A$ with ample normal bundle such that $Y$ has Hilbert polynomial $P$ with respect to $\omega_Y \otimes f^\ast L$, conditions~\ref{Eq:FirstNumericalCondition} and~\ref{Eq:NumericalConditions} hold, and $f(Y)$ is neither isogenous to a symmetric power of a curve nor isogenous to a product.
\end{theoremintro}

\Cref{IntroThmIntrinsicGeneral} is deduced from \cref{IntroThmExtrinsic} via the Shafarevich conjecture for abelian varieties, proved by Faltings~\cite{FaltingsMordell}. This reduction is similar to that in \cite[\S4.4]{KM}, and in \cref{sec:Applications} we only sketch the additional arguments required in our setting.
The proof of \cref{IntroThmExtrinsic} combines the big monodromy criterion established in the next section with the nondensity result of \cref{sec:Nondensity}, which generalizes~\cite[th.~D]{KM}. The latter is based on the Lawrence--Venkatesh method~\cite{LV}, as developed by Lawrence and Sawin~\cite{LS20}. In \cref{sec:Applications} we explain the main modifications needed to adapt the proof of \cite[\S4.3]{KM} to establish \cref{IntroThmExtrinsic}.

\subsection{Big monodromy} \label{sec:BigMonodromy} We now formulate the big monodromy criterion which is the main geometric input for the proof of \cref{IntroThmExtrinsic}. Let $S$ be a smooth connected variety over $k=\bbC$, and $A$ a complex abelian variety of dimension $g$. Suppose we are given a morphism $f \colon \cY \to A \times S$ such that the projection $\pi_S \colon \cY \to S$ is smooth with connected fibers of dimension $d$. Given an $n$-tuple $\underline{\chi} = (\chi_1, \dots, \chi_n)$ of characters
\[ \chi_i \colon \quad \pi_1(A(\bbC), 0) \;\too\; \bbC^\times\]
of the topological fundamental group of the abelian variety, let $L_{\underline{\chi}} = L_{\chi_1} \oplus \cdots \oplus L_{\chi_n}$ be the direct sum of the associated local systems of rank one on $A(\bbC)$, and consider the local system
\[ V_{\underline{\chi}} \;:=\; R^d \pi_{S \ast} \, \pi_A^* \, L_{\underline{\chi}} 
\]
where $\pi_A \colon \cY \to A$ is the projection. We want to ensure that for all $n$ and for sufficiently general tuples of characters, the Zariski closure of the image of the monodromy representation of $V_{\underline{\chi}}$ is as large as possible. In this case we say $\cY \to S$ has \emph{big monodromy for most tuples of torsion characters}; see~\cite[1.5]{KM} for the precise definition. As in loc.~cit.~it will be convenient to phrase our criterion in terms of the fiber over a geometric generic point $\bar{\eta}$ of $S$. We show:

\begin{bigmonodromyintro} Suppose $f_{\bar{\eta}} \colon \cY_{\bar{\eta}} \to A_{S, \bar{\eta}}$ is unramified with ample normal bundle, $d < (g - 1)/2$ and \cref{Eq:NumericalConditions} holds. Then, the following are equivalent:
\begin{enumerate}
\item $f_{\bar{\eta}}$ is birational onto its image, and this image is nondivisible, not constant up to translation, not a product and not a symmetric power of a curve;\smallskip
\item $f \colon \cY \to A \times S$ has big monodromy for most tuples of torsion characters.
\end{enumerate}
\end{bigmonodromyintro}

Here the subvariety $f(\cY_{\bar{\eta}})\subset A_{S, \bar{\eta}}$ is said to be \emph{constant up to translation} if it is a translate $Y + a$ of a complex subvariety $Y \subset A$  by a point $a \in A(\bar{\eta})$. If $f(\cY)_{\bar{\eta}}$ is nondivisible, then by~\cite[cor.~4.8]{JKLM} it is constant up to translation if and only if the family $f(\cY) \to S$ is isotrivial. Also note that the morphism $f_{\bar{\eta}}$ is finite, hence it is is birational if and only if $\cY_{\bar{\eta}}$ is the normalization of its image in $A_{S, \bar{\eta}}$.  \medskip

As in \cite[1.1]{JKLM}, if $f(\cY_{\bar{\eta}})$ is divisible, or constant up to translation, or a product, or a symmetric power of a curve, then $V_{\underline{\chi}}$ does not have big monodromy for most tuples of torsion characters. If $f_{\bar{\eta}}$ is not birational, then the pushforward to $A_{S,\bar{\eta}}$ of the constant sheaf on $\cY_{\bar{\eta}}$ has several direct summands, so the monodromy decomposes in block matrices and cannot be big. This proves (2) $\Rightarrow$ (1). \medskip

The nontrivial implication (1) $\Rightarrow$ (2) is the one relevant for \cref{IntroThmExtrinsic} and the main content of the big monodromy criterion. By \cite[th. 4.10]{JKLM}, it suffices to show that the subvariety $f(\cY_{\bar{\eta}}) \subset A_{S, \bar{\eta}}$ has big Tannaka group with respect to the convolution of perverse sheaves, 
which we will explain in section~\ref{sec:BigTannakaIntro}.

\subsection{Nondensity of integral points} \label{sec:Nondensity}

\Cref{IntroThmExtrinsic} is based on a nondensity criterion for integral points generalizing~\cite[th.~D]{KM}. Given a variety $S$ over the number field~$K$ we say that its \emph{integral points are not Zariski dense} if for any finite extension~$K'$ of~$K$, any finite set of primes~$\Sigma'$ in~$K'$ and any flat scheme~$\cS$ over~$\cO_{K', \Sigma'}$ with generic fiber~$S_{K'}$, the subset \[\cS(\cO_{K', \Sigma'}) \; \subset \; S(K')\]
is not Zariski dense in~$S_{K'}$. To obtain a criterion for this non-density, assume that $S$ is smooth and geometrically connected and that there exists an abelian variety $A$ over $K$ and a morphism
\[ f\colon \cY \to A\times S \]
such that  the projection $\pi_S \colon \cY \to S$ is smooth projective with geometrically connected fibers of dimension $d$. Let $\cX = f(\cY)$ be the image of $f$ and $\tilde{\cX} \to \cX$ its normalization. Since $\cY$ is smooth, the morphism $f$ then factors through $\tilde{\cX}$. Let $\cX = f(\cY)$ be the scheme theoretic image of $f$. Since $\cY$ is normal as a scheme, the morphism~$f$ factors through the normalization $\tilde{\cX}$ of $\cX$. The situation is summarized in the following diagram:
\[
\begin{tikzcd}
\ &&&\\
\cY \ar[r] \ar[rrr, phantom, "\scriptstyle{f}", yshift=5.7ex]
\arrow[u, phantom, ""{coordinate, name=X}]
\arrow[rrr,
to path={ -- ([xshift=0]\tikztostart.north)
|- ([xshift=0]X) \tikztonodes
-| ([xshift=0]\tikztotarget.north)}] 
& \tilde{\cX}  \ar[r]
 & \cX \arrow[r, hookrightarrow]
&   A \times S
\end{tikzcd}
\]
In~\cref{proof-th-D} we show:

\begin{theoremintro} \label{thm:non-density} Suppose that
\begin{itemize}
\item there is a nonempty open subset $S' \subset S$ over which $\cY \to \tilde{\cX}$ is \'etale and the family $\tilde{\cX} \to A \times S$ has big monodromy for most tuples of torsion characters;\smallskip
\item the fibers of $\cY \to S$ satisfy \cref{Eq:FirstNumericalCondition}.
\end{itemize}
Then integral points of $S$ are not Zariski dense.
\end{theoremintro}

Note that the morphism $\tilde{\cX} \to S$ is smooth over such an open subset $S'$, since $\cY$ is smooth over $S$ and $\cY \to \tilde{\cX}$ is finite \'etale over $S'$.\medskip

The flexibility of the above statement, as opposed to \cite[theorem D]{KM}, lies in the possibility of shrinking $S'$. This allows us to avoid reducing to the situation in which the fibers of $\cX \to S$ have trivial stabilizer in $A$. More crucially, it allows us to avoid reducing to the case in which $f$ is birational, which would require knowing that a canonically polarized variety has good reduction whenever one of its finite \'etale cover has (after possibly inverting finitely many primes depending only on numerical data associated with the variety and the cover). The fact that we can prove the nondensity of integral points of $S$ while imposing the big monodromy assumption only over $S'$ relies crucially on the smoothness of the family $\cY$ over the whole of $S$.

\subsection{Big Tannaka groups} \label{sec:BigTannakaIntro} 
We reset notation and let $A$ be an abelian variety of dimension $g$ over an algebraically closed field $k$ of characteristic zero. Our results hold both in the algebraic and in the analytic framework. In the former, we consider perserve sheaves for the \'etale topology on $A$ with coefficients in $\bbF = \bar{\bbQ}_\ell$ for a prime~$\ell$; in the latter, we assume $k = \bbC$ and consider perverse sheaves for the classical topology on $A(\bbC)$ with coefficients in $\bbF = \bbC$. For background about perverse sheaves we refer to~\cite{BBDG}.\medskip

Let $Y$ be a smooth projective variety of dimension $d$ and $f \colon Y \to A$ a semismall morphism in the sense that $\dim Y\times_A Y = \dim Y$. The constant sheaf $\delta_Y := \bbF_Y[d]$ on~$Y$ placed in degree $-d$ is a perverse sheaf, and the semismallness of $f$ then implies that $P=f_\ast \delta_Y$ is also perverse. The group law on~$A$ induces a convolution product on perverse sheaves, and $P$ generates a neutral Tannaka category $\langle P \rangle$ with respect to convolution~\cite[3.1]{JKLM}. We fix a fiber functor
\[ \omega \colon \qquad \langle P \rangle \; \too \; \Vect(\bbF)\]
and consider the associated Tannaka group
\[ G_{Y, \omega} \; := \; \Aut^{\otimes}(\omega) \; \subset \; \GL(V) \qquad \textup{where} \qquad V \; := \; \omega(P).\]
We say that $f \colon Y \to A$ is \emph{symmetric up to translation} if there is an involution $\iota$ of~$Y$ and a point $a \in A(k)$ such that $f(\iota(x)) = a - f(x)$ for all $x \in Y$. In this case, the group $G_{Y, \omega}$ preserves a bilinear form $\theta \colon V \otimes V \to \bbF$
which is symmetric if $d$ is even and alternating otherwise. The group $G_{Y, \omega}$ is then said to be \emph{big} if the derived subgroup $G_{Y, \omega}^\ast$ of its connected component is
\[
G_{Y, \omega}^* \;=\;
\begin{cases}
\SL(V) & \textup{if $Y$ is not symmetric up to translation}, \\
\SO(V, \theta)  & \textup{if $Y$ is symmetric up to translation and $d$ is even},  \\
\Sp(V, \theta)  & \textup{if $Y$ is symmetric up to translation and $d$ is odd}.
\end{cases}
\]
If $f$ is finite birational, which is the case relevant here, then $Y$ is the normalization of its image $X \subset A$ and therefore $P=\delta_X$ is the perverse intersection complex on this image, in which case $G_{Y, \omega}$ can be thought as the Tannaka group of $X$.

\begin{bigtannakaintro} Suppose that $f$ is unramified with ample normal bundle, $d < (g - 1)/2$ and \cref{Eq:NumericalConditions} holds. Then, the following are equivalent:
\begin{enumerate}
\item $f$ is birational onto its image, and this image is nondivisible, not a product and not a symmetric power of a curve;\smallskip
\item $G_{Y, \omega}$ is big.
\end{enumerate}
\end{bigtannakaintro}

Similarly to the big monodromy criterion, the implication (2) $\Rightarrow$ (1) is easy and the main task is to prove (1)~$\Rightarrow$~(2). We follow the strategy explained in the introduction of \cite{JKLM}:\medskip

The first step is to ensure that the algebraic group $G^\ast_{Y, \omega}$ is simple. For this we do not need to assume $f$ is unramified: in \cref{Thm:SimpleTannakaGroups} we show under very mild conditions that the Tannaka group of a subvariety $X\subset A$ fails to be simple only if the subvariety is a product. We are then in good shape since~$V$ is known to be a minuscule representation of $G_{Y, \omega}^\ast$ by~\cite[cor. 5.15]{JKLM} and there are very few minuscule representations of simple algebraic groups. Except for the standard representations of classical groups (which lead to big Tannaka groups), we are left with wedge powers of the standard representation of~$\SL_n$, spin representations, and the smallest irreducible representations of the exceptional groups~$E_6$ and~$E_7$. In \cref{Thm:SmallWedgePowersAreSumsOfCurves} we show that wedge powers occur if and only if $Y$ is the symmetric power of a smooth curve and $f$ is a closed embedding; in \cref{Thm:SmallSpinDoNotExist} we exclude spin representations, and in \cref{Thm:NoExceptionals} we rule out the groups $E_6$ and $E_7$. 

\subsection*{Acknowledgements} We thank Olivier Benoist for suggesting us the motivating example in the introduction. This project has been carried out during a Research In Paris and we thank the Institut Henri Poincar\'e for their hospitality. 

\subsection*{Conventions}
By a variety over a field $k$ we mean a separated $k$-scheme of finite type. Subvarieties are always taken to be closed. We write $\tilde{X}$ for the normalization of a variety $X$ and say that $X$ has smooth unramified normalization if $\tilde{X}$ is smooth and the normalization morphism $\tilde{X} \to X$ is unramified. 
If this is the case and if~$X\subset W$ is given as a subvariety of a smooth projective variety $W$, we say that~$\tilde{X}$ has ample normal bundle if the unramified morphism $\tilde{X} \to W$ does.

\section{Proof of theorems \ref{IntroThmIntrinsicGeneral} and \ref{IntroThmExtrinsic}}
\label{sec:Applications}

\subsection{Basic properties} \label{sec:BasicProperties} Let $f \colon Y \to A$ be a morphism between a smooth projective variety $Y$ and an abelian variety $A$ over a field $k$. By definition $f$ is unramified if and only if the cotangent map $\Lie A \otimes_k \cO_Y \to \Omega^1_Y$ is surjective. If this is the case, then the cotangent bundle of $Y$ is globally generated. For instance, if
\[ f = \alb_y \colon \quad  Y \; \too \; A \; = \; \Alb(Y)\]
is the Albanese morphism associated with a point $y \in Y(k)$, then the cotangent map is the evaluation morphism $\rH^0(Y, \cO_Y) \otimes_k \cO_Y \to \Omega^1_Y$. Therefore the Albanese morphism is unramified if and only if the cotangent bundle is globally generated. 

\begin{lemma} \label{Lem:GeneralTypeCanonicallyPolarized} Let $f \colon Y \to A$ be an unramified morphism where $Y$ is a smooth projective variety and $A$ an abelian variety. Then, the normalization $\tilde{X}$ of $X=f(Y)$ is smooth and the following conditions are equivalent:
\begin{enumerate}
\item $Y$ is canonically polarized; \smallskip
\item $\tilde{X}$ is canonically polarized; \smallskip
\item $Y$ is of general type; \smallskip
\item $\tilde{X}$ is of general type; \smallskip
\item $X$ is of general type.
\end{enumerate}
\end{lemma}

\begin{proof} First of all, the morphism $g \colon Y \to \tilde{X}$ induced by $f$ is unramified, hence \'etale by \cite[\href{https://stacks.math.columbia.edu/tag/0BTF}{Lemma 0BTF}]{stacks-project}, thus $\tilde{X}$ is smooth. The \'etale morphism $g$ induces an isomorphism between the canonical bundle of $Y$ and the pull-back of $\tilde{X}$, yielding the following implications:
\[ 
(1) \iff (2) \implies (3) \iff (4) \iff (5). 
\]
It suffices to show that if $Y$ is of general type, then the canonical bundle $\cK_Y$ of~$Y$ is ample. Since $Y$ is smooth, being of general type means that its canonical bundle~$\cK_Y$ is big. It is also nef, because $\Omega^1_Y$ is globally generated. In this case the non-ample locus $Z : = \rB_+(\cK_Y)$ is covered by rational curves; see for instance \cite[Corollary~A]{BoucksomPacienza}. Since there are no rational curves on $A$ and $f$ is finite, $Z$ must be finite. Then by \cite[Theorem A]{BoucksomCacciola} the morphism $h_d \colon Y \to \bbP(\rH^0(Y, \cK_Y^{\otimes d})^\vee)$ is an isomorphism outside $Z$ for any integer $d \ge 1$ divisible enough. In other words~$h_d$ is finite and $\cK_Y = h_d^\ast \cO(1)$ is ample.
\end{proof}

\subsection{Products and symmetric powers loci}
Let $S$ be a smooth variety over a field $k$ of characteristic $0$ and $\cA \to S$ a projective abelian scheme of relative dimension $g$. Let $\cY \to S$ be a smooth projective morphism with geometrically connected fibers of dimension $d$. Assume that we have an unramified morphism 
\[ f \colon \cY \to \cA\]
so we can talk about the normal bundle of $f$. Consider the locus 
\[ \Pi \subset S \quad \text{resp.} \quad  \Sigma \subset S\] 
made of images of geometric points $\bar{s} \to S$ such that $f(\cY_{\bar{s}})$ is isogenous to a product  resp. isogenous to a symmetric power of a curve in the sense of \cref{sec:ExtrinsicResults}.

\begin{theorem} \label{constructibility-product-and-sym-powers-loci} Suppose that the normal bundle of $f$ is ample. If $d < g-1$, then~$\Pi$ is constructible and, if moreover $d < (g - 1)/2$, then $\Sigma$ is constructible.
\end{theorem}

\begin{proof} 
Let $\cX=f(\cY)$ be the image of $f$. We first reduce to the case where, for every $s\in S$, the variety $\cX_s\subset\cA_s$ is nondivisible and the morphism
$
f_s\colon\cY_s\to\cX_s
$
is birational.\medskip

First, let $\cG\subset\cA$ be the group subscheme whose fibre over $s\in S$ is the stabilizer of $\cX_s\subset\cA_s$. Since the normal bundle of $f$ is ample, the variety $\cX_s$ is of general type, and therefore the subgroup $\cG_s$ is finite. It follows that $S$ admits a finite stratification by constructible subsets $S'\subset S$ such that $\cG_{S'}\to S'$ is finite \'etale. Treating each stratum separately, we may therefore assume from the outset that $\cG\to S$ is finite \'etale. Replacing $\cA$ by the quotient abelian scheme $\cA/\cG$, we may furthermore assume that $\cX_s$ is nondivisible for every $s\in S$.\medskip

Next, let $\tilde{\cX}$ denote the normalization of $\cX$. Since $\cY$ is smooth and $f$ is unramified, the morphism $f$ factors through a finite \'etale morphism
$
\cY \to \tilde{\cX},
$
so that $\tilde{\cX}\to S$ is smooth by \cite[\href{https://stacks.math.columbia.edu/tag/0BTF}{Lemma 0BTF}]{stacks-project}. By construction, the subsets $\Pi$ and $\Sigma$ are unchanged if we replace $\cY$ by $\tilde{\cX}$. Finally, although the normalization morphism
$
\tilde{\cX} \to \cX
$
is birational, its restriction to a fibre need not be birational. It therefore suffices to replace $S$ by the image under $\cY\to S$ of the open subset on which the morphism
$
\cY=\tilde{\cX} \to \cX
$
is an isomorphism.\medskip

From now on, we assume that for every $s \in S$ the variety $\cX_s \subset \cA_s$ is nondivisible and the morphism
$
f_s \colon \cY_s \to \cX_s
$
is birational, so that $\cY_s$ is the normalization of $\cX_s$. As in \cite[sect.~3.2]{JKLM}, consider the Tannaka group
\[
G_s \;=\; G_{\cX_{\bar{s}},\omega} \;\subset\; \GL(V_s),
\]
where $\omega$ is any fiber functor on the Tannakian category generated by the constant perverse sheaf on $\cX_{\bar{s}} \subset \cA_{\bar{s}}$, and where $V_s \in \Rep(G_s)$ is the corresponding faithful representation. By the invariance of Tannaka groups under extensions of algebraically closed fields \cite[cor.~4.4]{JKLM}, the isomorphism class of the pair $(G_s,V_s)$ depends only on the point $s \in S$, and not on the choice of the geometric point $\bar{s}$ above it. The derived subgroup of the identity component,
\[
G_s^\ast=[G_s^\circ,G_s^\circ],
\]
is a connected semisimple group. Moreover, its isomorphism class is unchanged if the subvariety $\cX_{\bar{s}} \subset \cA_{\bar{s}}$ is replaced by a translate \cite[lemma~4.3.2]{KraemerMicrolocalII}. It follows from \cite[prop.~7.4]{KWGeneric} that these groups define a constructible stratification of the base: Given a connected semisimple group $H$ together with a representation $W$, the locus
\[
\{\,s\in S \mid (G_s^\ast,V_s)\iso(H,W)\,\}
\;\subset\;
S
\]
is constructible. In fact, the statement of \cite[prop.~7.4]{KWGeneric} only asserts the existence of a constructible stratification according to the isomorphism class of $G_s^\ast$, but the proof proceeds by constructing a stratification as above.
\medskip

The constructibility of $\Pi$ and $\Sigma$ now follows immediately. Indeed, if $d<g-1$, the subvariety $\cX_s$ is a product if and only if the algebraic group $G_s^\ast$ is not simple, by \cref{Thm:SimpleTannakaGroups}. Similarly, if $d<(g-1)/2$, the subvariety $\cX_s$ is a symmetric power of a curve if and only if the pair $(G_s^\ast,V_s)$ is isomorphic to to a special linear group acting on an exterior power by \cref{Thm:SmallWedgePowersAreSumsOfCurves}.
\end{proof}

\subsection{Proof of \cref{IntroThmExtrinsic}} 
We return to the notation of \cref{IntroThmExtrinsic}. Let~$A$ be an abelian variety of dimension~$g$ over~$K$. By enlarging~$\Sigma$, we may assume that it contains all primes of bad reduction of~$A$ and that~$\cO_{K,\Sigma}$ is a principal ideal domain. Under this assumption,~$A$ extends to an abelian scheme~$\cA$ over~$\cO_{K,\Sigma}$. Moreover, the ample line bundle~$L$ extends to a relatively ample line bundle~$\cL$ on~$\cA$.\medskip

Fix~$P \in \bbQ[z]$ of degree~$d < (g-1)/2$ and an ample line bundle~$L$ on~$A$. By standard Hilbert scheme arguments, the Hilbert polynomial of the canonical bundle of the varieties appearing in the statement takes only finitely many values. By considering each of these values separately, we may therefore assume that this Hilbert polynomial is a fixed polynomial~$\Phi \in \bbQ[z]$. As in \cite[\S4.1]{KM}, choose an integer~$n \ge 1$ such that, for every canonically polarized variety~$Y$ over a field~$k$ with Hilbert polynomial~$\Phi$ and every integer~$m \ge n$, the line bundle~$\omega_Y^{\otimes m}$ is very ample and satisfies~$\rH^i(Y,\omega_Y^{\otimes m}) = 0$ for all~$i \ge 1$. Then the functor associating to a Noetherian~$\cO_{K,\Sigma}$-scheme~$S$ the set of triples consisting of
\begin{itemize}
\item a family~$\pi \colon Y \to S$ of canonically polarized varieties with Hilbert polynomial~$\Phi$,
\item an isomorphism of~$\cO_S$-modules
$
\cO_S^{\Phi(n)} \simeq \pi_\ast \omega_{Y/S}^{\otimes n},
$
and
\item a morphism~$f \colon Y \to \cA_S$ such that the fibres of~$\pi$ have Hilbert polynomial~$P$ with respect to~$\omega_{Y/S}\otimes f^\ast\cL$,
\end{itemize}
is representable by a quasiprojective scheme~$\cH$ over~$\cO_{K,\Sigma}$. Indeed, this follows by combining \cite[Lemma~4.1]{KM} with the standard theory of the scheme of morphisms \cite[\S5.6.2]{FGAExplained}.\medskip

The varieties~$Y$ endowed with an unramified morphism~$f \colon Y \to A$ that we are interested in naturally define~$\cO_{K,\Sigma}$-valued points of~$\cH$. Indeed, the morphism~$f$ extends uniquely to a morphism~$\cY \to \cA$, where~$\cY$ is the canonically polarized model of~$Y$; such an extension exists by the smoothness of~$\cY$ and \cite[Cor.~8.4.6]{NeronModels}. Moreover, since every ideal of~$\cO_{K,\Sigma}$ is principal, the module~$\rH^0(\cY,\omega_{\cY}^{\otimes n})$ is free, so we may choose a basis. Write~$H$ for the generic fibre of~$\cH$. By \cref{constructibility-product-and-sym-powers-loci}, the locus
\[
H' \subset H
\]
where the universal morphism is unramified with ample normal bundle, the fibres of its image are neither isogenous to a product nor to a symmetric power of a curve, and the fibres of the universal family satisfy conditions~\ref{Eq:FirstNumericalCondition} and~\ref{Eq:NumericalConditions}, is constructible. By the above discussion, it therefore suffices to show that the varieties parametrized by
\[
H'(K)\cap \cH(\cO_{K,\Sigma})
\]
are finite up to translation by points of~$A(K)$. Let~$\cT\subset\cH$ be an integral closed subscheme of the Zariski closure of~$H'(K)\cap\cH(\cO_{K,\Sigma})$ in~$\cH$ whose generic fibre
\[
T\subset H
\]
contains~$T\cap H'(K)\cap\cH(\cO_{K,\Sigma})$ as a Zariski-dense subset. Since~$\cT$ is closed in~$\cH$, we have
\[
\cT(\cO_{K,\Sigma})=\cH(\cO_{K,\Sigma})\cap T(K).
\]
By Noetherian induction, it is enough to show that there exists a nonempty open subset
\[
U\subset T':=T\cap H'
\]
such that the varieties parametrized by~$U(K)\cap\cT(\cO_{K,\Sigma})$ are finite up to translation by points of~$A(K)$. To this end, let
\[
\pi\colon S\to T
\]
be a desingularization, that is, a proper birational morphism with~$S$ smooth over~$K$. After enlarging~$\Sigma$, we may assume that~$\pi$ extends to a proper birational morphism~$\cS\to\cT$ with~$\cS$ flat and separated over~$\cO_{K,\Sigma}$. By construction,~$\cT(\cO_{K,\Sigma})$ is Zariski dense in~$T$, and therefore~$\cS(\cO_{K,\Sigma})$ is Zariski dense in~$S$. Let
\[
f\colon \cY\to A\times S
\]
be the pullback of the universal family over~$H$. By construction,~$f$ is unramified over an open subset~$S'$ containing~$\pi^{-1}(T')$. Let
\[
f(\cY)\subset A\times S
\]
denote the image of~$f$, and let~$G$ be the stabilizer of the fibre of~$f(\cY)$ over the generic point~$\eta\in S$. Since~$S(K)$ is Zariski dense in~$S$, it is easy to see that~$G$ is a subgroup of~$A$ defined over~$K$. Let~$p\colon A\to B:=A/G$ be the quotient morphism, let~$\cX$ be the image of the composite morphism
\[
g\colon
\cY
\stackrel{f}{\longrightarrow}
A\times S
\stackrel{p\times\id}{\longrightarrow}
B\times S,
\]
and let~$\tilde{\cX}\to\cX$ be its normalization. By construction, the generic fibre~$\cX_\eta$ is a nondivisible subvariety of~$B_{S, \eta}$. Moreover, since~$f$ is unramified over~$S'$, so is~$g$. Hence the induced morphism~$\cY\to\tilde{\cX}$ is finite \'etale over~$S'$ by \cite[\href{https://stacks.math.columbia.edu/tag/0BTF}{Lemma 0BTF}]{stacks-project}.\medskip

The key point is that the fiber~$\cX_{\bar{\eta}}$ over any geometric generic point~$\bar{\eta}$ of~$S$ is constant up to translation as a subvariety of~$B_{S, {\bar{\eta}}}$. Indeed, otherwise \cref{thm:non-density} together with the Big Monodromy Criterion of \cref{sec:BigMonodromy}, would imply that~$\cS(\cO_{K,\Sigma})$ is not Zariski dense, a contradiction. Since~$\cX_{\bar{\eta}}$ is constant up to translation, it follows from \cite[Cor.~4.8]{JKLM} that, over~$S'$, the family~$\cX$ is obtained by translating a fixed subvariety~$X\subset B$ along a morphism~$b\colon S'\to B$. In particular,
\[
\cX_s=X+b(s)
\qquad\text{for all } s\in S'(K).
\]
Since~$\cY\to\tilde{\cX}$ is \'etale over~$S'$, it follows that for
$
s\in S'(K)\cap\cS(\cO_{K,\Sigma})
$
the pairs~$(\cY_s,g_s)$ fall into only finitely many equivalence classes up to translation by points of~$B(K)$; see, for instance, \cite[Lemma~4.5]{KM}. The projection~$A(K)\to B(K)$ has finite cokernel, thus the same conclusion holds true after replacing~$g_s$ by~$f_s\colon\cY_s\to A$ and considering translation by points of~$A(K)$ instead. Therefore, the open subset~$U\subset T'$ on which the birational morphism~$\pi\colon S'\to T'$ is an isomorphism has the required property.\qed

\subsection{Proof of \cref{IntroThmIntrinsicGeneral}} The deduction of \cref{IntroThmIntrinsicGeneral} from \cref{IntroThmExtrinsic} is essentially the same as that of \cite[4.4]{KM}. The only substantive modification is that here one has to work with the locus where the Albanese morphism is unramified, rather than a closed embedding. Moreover, varieties with globally generated cotangent bundle whose image in the Albanese variety is a symmetric power of a curve in the sense of \cref{sec:ExtrinsicResults} are treated by the following lemma:

\begin{lemma} Fix $\chi \in \bbZ$. Up to isomorphism, there are only finitely many canonically polarized varieties over $K$ with good reduction and topological Euler characteristic $\chi$ that over an algebraic closure of $K$ become isomorphic to a finite \'etale cover of a symmetric power of a curve of genus $\ge 2$.
\end{lemma}

The previous statement is equivalent to \cite[cor.~4.8]{KM}, since the Albanese morphism of the varieties under consideration is a closed embedding with ample normal bundle. This, in turn, follows from the classical fact that, given a smooth projective curve $C$ of genus $g \ge 2$ over an algebraically closed field of characteristic~$0$, then the \'etale fundamental group of $X=\Sym^n C$ is isomorphic to the abelianization of that of $C$, provided that $n\ge 2$. Geometrically, this identification is realized as follows. A finite \'etale abelian cover $C' \to C$ is obtained by pulling back $C$ along a suitable isogeny $A' \to A=\Pic^0(C)$, where $C$ is embedded into $A$ via the Abel--Jacobi map. The corresponding finite \'etale cover of $X$ is then given by the following cartesian square:
\[
\begin{tikzcd}
X' \ar[d] \ar[r] & A' \ar[d] \\
X=\Sym^n C \ar[r] & A=\Pic^0(C),
\end{tikzcd}
\]
where the bottom horizontal morphism is the sum map $(x_1,\dots,x_n)\mapsto x_1+\cdots+x_n$. Now, the variety $X'$ is of general type if and only if $X$ is, which in turn is equivalent to $n<g$. Indeed, for $n\ge g$, the variety $X$ is birational to a projective bundle over $A$, and hence is not of general type. If $X'$ is of general type, then  $X\hookrightarrow A$ is a closed embedding and the same holds for $X'\hookrightarrow A'$. Consequently, since the normal bundle of $X\hookrightarrow A$ is then ample, the normal bundle of $X'\hookrightarrow A'$ is ample as well.

\section{Proof of \cref{thm:non-density}} \label{proof-th-D}

In this section we prove the nondensity criterion for integral points stated in \cref{sec:Nondensity}.  By extending scalars and taking a model for the base of the family, \cref{thm:non-density} reduces to \cref{nondensity-with-fixed-model} below, whose proof will be the main goal of this section.

\subsection{Statement}
Let $K$ be a number field and $\Sigma$ a finite set of primes. Let $A$ be an abelian scheme over $\cO_{K, \Sigma}$ of relative dimension $g$, and let $S$ be a smooth separated scheme over $\cO_{K, \Sigma}$ with geometrically connected fibers. Consider a morphism $f \colon \cY \to A \times S$ of $\cO_{K, \Sigma}$-schemes such that the projection $\cY \to S$ is smooth projective with geometrically connected fibers of dimension $d$. Let $\cX = f(\cY)$ be the scheme-theoretic image of $f$. Since $\cY$ is normal as a scheme, the morphism~$f$ factors through the normalization $\tilde{\cX}$ of $\cX$. The situation is summarized in the following diagram:
\[
\begin{tikzcd}
\ &&&\\
\cY \ar[r, "\pi"] \ar[rrr, phantom, "\scriptstyle{f}", yshift=5.7ex]
\arrow[u, phantom, ""{coordinate, name=X}]
\arrow[rrr,
to path={ -- ([xshift=0]\tikztostart.north)
|- ([xshift=0]X) \tikztonodes
-| ([xshift=0]\tikztotarget.north)}] 
& \tilde{\cX}  \ar[r]
\arrow[d, phantom, ""{coordinate, name=Z}]
\arrow[rr,
to path={ -- ([xshift=0]\tikztostart.south)
|- ([xshift=0]Z) \tikztonodes
-| ([xshift=0]\tikztotarget.south)}] 
 & \cX \arrow[r, hookrightarrow] \ar[d, phantom, "\scriptstyle{h}", near start]
&   A \times S
\\
& \ & \ &  
\end{tikzcd}
\vspace{-7pt}
\]
\Cref{thm:non-density} is a consequence of the following result:

\begin{theorem} \label{nondensity-with-fixed-model} Suppose that
\begin{itemize}
\item there is a nonempty open subset $S' \subset S$ over which $\pi$ is \'etale and $h$ has big monodromy for most tuples of torsion characters;
\item the fibers of $\cY \to S$ satisfy \cref{Eq:FirstNumericalCondition}.
\end{itemize}
Then $S(\cO_{K, \Sigma})$ is not Zariski dense in $S$.
\end{theorem}

The proof of this statement will occupy the rest of this section.

\subsection{Cohomological good reduction} 
Under the hypotheses of \cref{thm:non-density}, the family $\tilde{\cX} \to S$ is smooth over $S'$, since $\pi \colon \cY \to \tilde{\cX}$ is finite \'etale above $S'$. Outside~$S'$, however, the morphism $\tilde{\cX} \to S$ need not be smooth. Nevertheless, the key point is that, from the point of view of cohomology, it behaves as though it were.\medskip

To make this precise, let $\pi \colon Y \to Z$ be a finite \'etale morphism of smooth projective varieties over a field $k$ of characteristic $0$. Let $E$ be an \'etale local system on~$Z$ with coefficients in $\bbQ_p$ such that, after a finite extension of $k$ and extension of scalars to $\bar{\bbQ}_p$, it decomposes as a direct sum of rank one local systems whose $n$-th tensor powers are the constant sheaf, for some integer $n \ge 1$. By the projection formula, together with the trace morphism
\(
\pi_\ast\pi^\ast E \to E,
\)
the natural morphism
\(
E \into \pi_\ast\pi^\ast E
\)
admits a splitting. Thus $E$ is a direct summand of $\pi_\ast\pi^\ast E$. Consequently, if $\bar{k}$ is an algebraic closure of $k$, then
\[
\rH^i_\et(Z_{\bar{k}}, E)
\quad
\text{is a direct summand of}
\quad
\rH^i_\et(Y_{\bar{k}}, \pi^\ast E)
\]
as a representation of the absolute Galois group $\Gamma = \Gal(\bar{k}/k)$. In particular, if $k$ is a number field and $Y$ has good reduction outside a finite set of primes $\Sigma$ containing those above $p$, then the representation
$
\rH^i_\et(Z_{\bar{k}}, E)
$
is unramified and pure of weight $i$ outside $\Sigma$. \medskip

These considerations admit analogous formulations in the other cohomological realizations. For instance, on the de Rham side, let $(\cE, \nabla)$ be a flat vector bundle on $Z$ such that, after passing to a finite extension of $k$, it decomposes as a direct sum of flat line bundles whose $n$-th tensor powers are trivial. Then
\[
\rH^i_\dR(Z/k, \cE, \nabla)
\quad \text{is a direct summand of} \quad
\rH^i_\dR(Y/k, \pi^\ast\cE, \pi^\ast\nabla),
\]
and the Hodge filtration on $\rH^i_\dR(Y/k, \pi^\ast\cE, \pi^\ast\nabla)$ admits that of $\rH^i_\dR(Z/k, \cE, \nabla)$ as a direct summand.\medskip

These direct summands are compatible under the $p$-adic comparison isomorphisms. Assume that $k$ is a finite extension of $\bbQ_p$, and let $E$ denote the local system of horizontal sections of $(\cE,\nabla)$. The $p$-adic \'etale--de Rham comparison theorem identifies
\begin{align*}
\rD_{\dR}\bigl(\rH^i_\et(Z_{\bar{k}},E)\bigr)
&\iso
\rH^i_\dR(Z/k,\cE,\nabla), \\
\rD_{\dR}\bigl(\rH^i_\et(Y_{\bar{k}},\pi^\ast E)\bigr)
&\iso
\rH^i_\dR(Y/k,\pi^\ast\cE,\pi^\ast\nabla),
\end{align*}
compatibly with the above decompositions and with the Hodge filtrations. Assume furthermore that the extension $k$ is unramified and that $Y$ has good reduction, i.e., that it is the generic fiber of a smooth projective scheme over the ring of integers of $k$. If $p\nmid n$, then $\rH^i_\et(Y_{\bar{k}},\pi^\ast E)$ is crystalline. Since the category of crystalline representations is closed under taking subquotients, the direct summand
$
\rH^i_\et(Z_{\bar{k}},E)
$
is crystalline as well. Using the compatibility between the functors $\rD_{\cris}$ and $\rD_{\dR}$, we conclude that the subspace
\[
\rH^i_\dR(Z/k,\cE,\nabla)
\hookrightarrow
\rH^i_\dR(Y/k,\pi^\ast\cE,\pi^\ast\nabla)
\]
is stable under the crystalline Frobenius acting on
$
\rH^i_\dR(Y/k,\pi^\ast\cE,\pi^\ast\nabla)$.\medskip

Going back to an arbitrary field of characteristic zero, suppose now that $Y$ and $Z$ vary in families. More precisely, let $\cZ \to S$ be a projective morphism, where $S$ is a smooth geometrically connected variety, and let $\pi \colon \cY \to \cZ$ be a morphism which is finite \'etale over a nonempty open subset $S' \subset S$, such that the composite morphism $\cY \to S$ is smooth and projective. Assume furthermore that the flat vector bundle $(\cE,\nabla)$ extends to a vector bundle with a connection on $\cZ$, still denoted $(\cE,\nabla)$, whose pullback to $\cY$ is integrable. Since $\cZ$ is not assumed to be smooth, this allows us to avoid defining integrability for connections on singular varieties. In the applications below, $(\cE,\nabla)$ will arise as the pullback of a flat bundle on an abelian scheme. As before, we assume that, after a finite extension of the base field, $(\cE,\nabla)$ decomposes as a direct sum of line bundles with connection whose $n$-th tensor powers are the trivial flat bundle. Over $S'$, the morphism $\pi\colon\cY'\to\cZ'$ is finite \'etale, so $\cE$ is a direct summand of $\pi_\ast\pi^\ast\cE$ as a flat bundle. Consequently,
\[
\cH^i_{\dR}(\cZ'/S',\cE,\nabla)
\hookrightarrow
\cH^i_{\dR}(\cY'/S',\pi^\ast\cE,\pi^\ast\nabla)
\]
is a direct summand of flat bundles, where both relative de Rham cohomology groups are equipped with their Gauss--Manin connections. Here $\cY'$ and $\cZ'$ denote the base changes of $\cY$ and $\cZ$ to $S'$. We claim that this direct summand extends uniquely across $S$. Indeed, let
\[
\cV :=
\cH^i_{\dR}(\cY/S,\pi^\ast\cE,\pi^\ast\nabla).
\]
The projector defining the direct summand over $S'$ is a horizontal section of $\cEnd(\cV)$ over $S'$. Since horizontal sections extend uniquely, it extends uniquely to a horizontal section of $\cEnd(\cV)$ over $S$. Being idempotent is a closed condition, so the extended endomorphism is again idempotent. Its image therefore defines a flat subbundle
\[
\cW \hookrightarrow \cV,
\]
which restricts over $S'$ to $
\cH^i_{\dR}(\cZ'/S',\cE,\nabla)$, whereas its kernel 
$\cW' \into \cV$
is the unique flat complement of $\cW$. 
The Hodge filtration on $\cW$ is a direct summand of the Hodge filtration on $\cV$. Indeed, this holds over $S'$, where $\cV$ is the de Rham cohomology of $\pi_\ast \pi^\ast \cE$ and $\cW$ is the de Rham cohomology of the direct summand $\cE \into \pi_\ast \pi^\ast \cE$. Since the locus of filtrations that split with respect to a fixed decomposition is closed in the corresponding flag variety, the same remains true over all of $S$. In particular, for every $s\in S(k)$, the formal period mapping for $\cV$ at $s$,
\[
\hat{\Phi}_s\colon
\Spec\widehat{\cO}_{S,s}
\longrightarrow
\Flag(\cV_s),
\]
factors through the closed subvariety of split filtrations,
\[
\Flag(\cW_s)\times\Flag(\cW_s')
\hookrightarrow
\Flag(\cV_s).
\]
Here $\widehat{\cO}_{S,s}$ denotes the completion of the local ring $\cO_{S,s}$ at its maximal ideal, and $\Flag(V)$ denotes the flag variety of a vector space $V$.\medskip

Finally, we conclude with a few words regarding $p$-adic period mappings. Suppose that $k$ is a finite unramified extension of $\bbQ_p$ with ring of integers $R$, that $S$ is a smooth separated scheme over $R$, and that $\pi \colon \cY \to \cZ$ is a morphism which is finite \'etale over a nonempty open subset $S' \subset S$, with $f \colon \cY \to S$ smooth and projective. Assume further that the Hodge-to-de Rham spectral sequence
\[ R^j f_\ast (\Omega^i_{\cY/S} \otimes \cE) \implies \cH_{\dR}^{i + j}(\cY/S; \cE, \nabla) = \rR f^{i+j}_\ast (\Omega^\bullet_{\cY/S} \otimes \cE) \]
degenerates, where $\Omega^\bullet_{\cY/S} \otimes \cE$ is the de Rham complex of $(\cE, \nabla)$, and that the quasicoherent sheaves
\[ \cV := \cH_{\dR}^q(\cY/S; \cE, \nabla) = \rR f^q_\ast (\Omega^\bullet_{\cY/S} \otimes \cE), \qquad R^j f_\ast (\Omega^i_{\cY/S} \otimes \cE), \]
are locally free of finite rank. Having fixed an $R$-valued point $s \in S(R)$, under the above hypotheses we obtain a $p$-adic period map
\[ \Omega \to \Flag(\cV_s)^\an, \]
where the superscript $\an$ indicates the analytification of an algebraic variety, and $\Omega \subset S_k^\an$ is the $p$-adic disk consisting of points of $S$ having the same reduction as $s$. As explained above, this period map takes values in the closed subvarieties parametrizing split filtrations. Hence, using the projection, we obtain the $p$-adic period map
\[ \Phi_p \colon \Omega \to \Flag(\cW_s) \]
for the direct factor $\cW$ of $\cV$ over $S_k$ extending $\cH^q_{\dR}(\cZ/S', \cE, \nabla)$. Note that $\cZ$ does not have good reduction. Thus, applying the same discussion to $\cZ$ would only yield convergence of the $p$-adic period map on the much smaller, possibly empty, $p$-adic disk around $s$ contained in $S'$.

\subsection{Proof of \cref{nondensity-with-fixed-model}} The proof of \cref{nondensity-with-fixed-model} follows the same lines as that of theorem~10.1 in \cite{KM}. Rather than going through the mass of details involved, we therefore only pinpoint the modifications needed. Roughly speaking, the main idea is to replace the smooth family $\cY \to S$, which fails to have big monodromy in the sense of \cref{sec:BigMonodromy} as soon as $\deg(\pi)>1$, by the possibly singular family $\tilde{\cX}\to S$. The crucial point is that, as explained in the previous section, the latter family behaves cohomologically as a smooth family.\medskip

To begin with the numerics, and keeping the same notation as in \cite[10.2]{KM}, we make the same choices, with $X$ now taken to be the fiber
of $\tilde{\cX}\to S$ over a geometric generic point $\bar{\eta}$ of $S$. In particular,
\[
e = (-1)^d \chi_{\top}(X),
\]
where~$\chi_{\top}(X)$ is the topological Euler characteristic of~$X$, and $H$ is the \emph{split} reductive group over~$\bbQ$ defined as
\[
H = 
\begin{cases}
\GL_e& \textup{if~$X$ is not symmetric up to translation,} \\
\GO_e & \textup{if~$X$ is symmetric up to translation and~$d$ is even,} \\
\GSp_e & \textup{if~$X$ is symmetric up to translation and~$d$ is odd}.
\end{cases}
\]
Moreover, for $\alpha = 0, \dots, d$ consider the Euler characteristic
\[ h_\alpha := (-1)^{d-\alpha} \chi(X, \Omega^\alpha_{X}) \]
and consider a point $y$ in the rational building of $H$ such that the induced filtration on $W$ satisfies
\[ \dim \gr^\alpha W = h_\alpha, \quad \alpha = 0, \dots, d.\]
The morphism $\pi$ is finite \'etale, thus
\[
 (-1)^d \,\chi(\Omega^d_{Y \times Y}) \;\le\; \tfrac{1}{2} \,\chi_{\top}(Y \times Y)
 \iff  (-1)^d \,\chi(\Omega^d_{X \times X}) \;\le\; \tfrac{1}{2} \,\chi_{\top}(X \times X),
 \]
where $Y = \cY_{\bar{\eta}}$ is the geometric generic fiber of $\cY$. The inequality on the left hand side always holds by \cite[lemma C.5]{KM} when $d$ is odd or $d = 2$, and holds by assumption when $d$ is even. In particular, if $P_y \subset H$ denotes the stabilizer of $y$, a direct computation shows that the following inequality holds:
\[ 2 \dim P^\ss_y \le \dim H + \rk H\]
where $P^\ss_y = P_y / \rad P_y$ is the semisimplification of $P_y$. With this notation, the integer $N$ is then chosen as in \cite[10.2]{KM}.\medskip

Let~$A^\natural$ be the moduli space of rank~$1$ connections on~$A$ and~$(\cL,\nabla)$ the universal rank~$1$ connection on~$A \times A^\natural$. Up to enlarging~$\Sigma$ we may assume that the Hodge to de Rham spectral sequence
\[ \rR^j (\pi_S \times \id_{A^\natural})_\ast (\Omega^i_{\cY \times A^\natural/S \times A^\natural} \otimes \cL) \implies \cH^{i + j}_{\dR}(\cY \times A^\natural/S \times A^\natural; \cL,\nabla)\]
degenerates and that the quasicoherent sheaves over~$A^\natural \times S$,
\[ \cH^i_{\dR}(\cY \times A^\natural/S \times A^\natural; \cL,\nabla), \quad \rR^j (\pi_S \times \id_{A^\natural})_\ast (\Omega^i_{\cY \times A^\natural/S \times A^\natural} \otimes \cL),\]
are locally free of finite rank for all~$i, j \in \bbN$, where~$\pi_S \colon \cY \to S$ is the projection. If~$\cX_{\bar{\eta}}$ is not symmetric up to translation, we may shrink~$S'$ suppose that no fiber of~$\cX$ over $S'$ is symmetric up to translation.\medskip

Let $\bar{K}$ be an algebraic closure of $K$ and $\Gamma := \Gal(\bar{K}/K)$ the absolute Galois group. Fix also a complex embedding $\sigma \colon \bar{K} \into \bbC$. After choosing a friendly place $v$ as in \cite[10.3]{KM}, fix an integer $r \ge 1$ not divisible by the rational prime $p$ below $v$, together with a character
\[
\chi_0 \colon A[r](\bar{K}) \too \bar{K}^\ast
\]
satisfying the following generic vanishing assumption:
\[
\rH^i\bigl(\tilde{\cX}_{s}(\bbC), \pi^\ast L_{\gamma\chi}\bigr) = 0
\quad \text{for all $\chi \in \Gamma.\chi_0$, $s \in S'_\sigma(\bbC)$, and $i \neq d$},
\]
where $\Gamma.\chi_0$ stands for the orbit of $\chi_0$ within the characters of $A[r](\bar{K})$ under the action of the absolute Galois group $\Gamma$.
We may further assume, by \cite[lemmas 4.8 and 4.9]{LS20}, that for every $\chi \in \Gamma.\chi_0$ there exist pairwise distinct characters
\[
\chi_1,\dots,\chi_N \in \Gamma_v \chi
\]
such that the family $\tilde{\cX}_{S',\sigma} \to A_\sigma \times S'_{\sigma}$ has big monodromy for the $N$-tuple of characters $\underline{\chi} = (\chi_1,\dots,\chi_N)$.\medskip

Needless to say, to show that the integral points of $S$ are not dense, it suffices to show that the subset
$
S'(K) \cap S(\cO_{K,\Sigma})
$
is not dense. Moreover, since $S(\bbF_v)$ is finite, it suffices to show that every set of the form
\[
F := S'(K) \cap S(\cO_{K,\Sigma}) \cap \red^{-1}(s_0)
\qquad \text{for some } s_0 \in S(\bbF_v)
\]
is not Zariski dense. Here $\bbF_v$ is the residue field at $v$ and 
$
\red \colon S(\cO_{K,\Sigma}) \to S(\bbF_v)
$
denotes the reduction map. The construction of the ``filtered $\Gamma$-'etale datum''
\[
D_s = (G_s, V_s, \rho_s, x_s)
\]
associated with any such $s \in F$ is then carried out in the same way as in \cite[10.3]{KM}, pulling back via the morphism $\tilde{\cX}_s \to A$ rather than the inclusion $\cX_s \into A$.\medskip

After this preparation, and in view of the discussion in the previous paragraph, the remainder of the proof unfolds with no substantial changes. We therefore refer the reader to \cite[10.4-10.6]{KM} for details. \qed

\section{Perverse sheaves and characteristic cycles} 

The rest of this paper will be concerned with Tannaka groups of perverse sheaves on abelian varieties. In this section we set up the general framework and recall a few facts about characteristic cycles that will be used in what follows.

\subsection{Setup} 
From now on $A$ will always denote an abelian variety of dimension~$g$ over an algebraically closed field $k$ of characteristic zero, and $X \subset A$ an integral subvariety of dimension $d$. We treat the algebraic and analytic framework on the same footing: in the former we use perverse sheaves for the \'etale topology on  $A$ with coefficients in $\bbF = \bar{\bbQ}_\ell$ for a prime $\ell$, in the latter we work over $k = \bbC$ and use perverse sheaves for the classical topology on $A(\bbC)$ with coefficients in $\bbF = \bbC$. We denote by
\[
 \delta_X \;\in\; \Perv(A, \bbF)
\]
the perverse intersection complex supported on $X\subset A$. By~\cite[prop.~3.1]{JKLM} the convolution powers of this intersection complex generate a neutral Tannaka category~$\langle \delta_X \rangle$. We fix a fiber functor
\[
 \omega\colon \quad \langle \delta_X \rangle \;\too\; \Vect(\bbF)
\]
and consider the associated Tannaka group $G_{X, \omega} := \Aut^\otimes(\omega) \subset \GL(\omega(\delta_X))$. This is a reductive group, and we denote the derived group of its connected component by
\[
 G_{X, \omega}^\ast \;:=\; [G_{X, \omega}^\circ, G_{X, \omega}^\circ].
\]

\subsection{Characteristic cycles} 
The above Tannaka groups are closely related to the conormal geometry of subvarieties. Recall that the \emph{conormal variety} to $X\subset A$ is defined as the Zariski closure in $T^\vee A=A\times \Lie(A)^\vee$ of the conormal bundle to the smooth locus of $X$. This is a conic Lagrangian subvariety of $T^\vee A$. We denote its projectivization by
\[
 \PLambda_X \; \subset \; A\times \bbP_A
 \quad \text{where} \quad 
 \bbP_A \;:=\; \bbP(\Lie(A)^\vee)
\]
and define the \emph{Gauss map} of the subvariety as the projection $\gamma_X\colon \PLambda_X \to \bbP_A$. In the context of perverse sheaves such conormal varieties arise naturally as irreducible components of characteristic cycles:

\medskip

Over $k=\bbC$ one may attach to any $P\in \Perv(A, \bbF)$ a \emph{characteristic cycle} $\CC(P)$, which is a finite formal sum of conormal varieties to the strata in a suitable Whitney stratification of the support $\Supp(P)$~\cite[def.~4.3.19]{DimcaSheaves}. We denote by $\cc(P)$ the cycle on $A\times \bbP_A$ which is obtained from the projectivization of the characteristic cycle by discarding any components with non-dominant Gauss map. If $X\subset A$ is a smooth subvariety of general type, then $\cc(\delta_X)=\PLambda_X$. We are interested in a class of subvarieties which may be singular but still satisfy this last condition. 

\medskip 

Over an arbitrary algebraically closed field $k$ of characteristic zero this leads to the following notion:

\begin{definition}
We say that a subvariety $X\subset A$ has \emph{integral characteristic cycle} if for any embedding $\sigma\colon k\hookrightarrow \bbC$, the associated complex subvariety $X_\sigma \subset A_\sigma$ has characteristic cycle
\[
 \cc(\delta_{X_\sigma}) \;=\; \PLambda_{X_\sigma}.
\]
\end{definition}

Note that this condition is independent of the chosen embedding $\sigma$: indeed, the characteristic cycle of the perverse intersection complex can be defined over~$k$ via Saito's theory for $\ell$-adic sheaves~\cite{SaitoCC}, and the base change of this cycle via $\sigma$ gives back the characteristic cycle in the above sense by~\cite[th.~1.2]{RaiComparison}. 

\medskip 

The reason for the above definition is that for any nondivisible subvariety $X\subset A$ with integral characteristic cycle, the Tannaka group $G_{X,\omega}^\ast$ acts via a minuscule representation on $\omega(\delta_X)$, see~\cite[cor.~1.10]{KraemerMicrolocalI}~\cite[cor.~5.15]{JKLM}. The main class of subvarieties with integral characteristic cycles to be considered in this paper will be the following:

\begin{definition}
We say that variety $X$ has \emph{smooth unramified normalization} if $\tilde{X}$ is smooth and the normalization morphism $\nu\colon \tilde{X} \to X$ is unramified. 
\end{definition} 

If this is the case and if $X\subset W$ is a subvariety of a smooth variety $W$, then the kernel
\[
 \cC_{\tilde{X}/W} \;=\; \ker\left(\nu^*(\Omega^1_W) \too \Omega^1_{\tilde{X}}\right)
\]
is a vector bundle whose associated projective bundle is the pullback of the conormal variety:
\[
 \bbP(\cC_{\tilde{X}/W}) \;=\; \tilde{X}\times_X \PLambda_X
\]
We say that $X\subset W$ has \emph{smooth unramified normalization with ample normal bundle} if it has smooth unramified normalization and the normal bundle $\cN_{\tilde{X}/W}=(\cC_{\tilde{X}/W})^\vee$ is ample. Let us now again take $W=A$ to be an abelian variety.

\begin{lemma} \label{Lem:IntegralCC}
For any subvariety $X\subset A$ with smooth unramified normalization, the following properties hold:\smallskip 
\begin{enumerate} 
\item $X$ has integral characteristic cycle and the morphism $\PLambda_X \to X$ has constant fiber dimension.\smallskip
\item $X$ is of general type if and only if the Gauss map $\gamma_X\colon \PLambda_X \to \bbP_A$ is dominant, in which case~$\gamma_X$ is generically finite of degree 
\[ \deg(\gamma_X) \;=\; \chi(\delta_X) \;=\; (-1)^d\, \chi_\top(\tilde{X})
\quad \text{where} \quad d \;=\; \dim X. \]
\item The following three properties are equivalent:\smallskip
\begin{itemize} 
\item The normal bundle $\cN_{\tilde{X}/ A}$ is ample.\smallskip
\item The Gauss map $\gamma_X \colon \PLambda_X \to \bbP_A$ is finite.\smallskip
\item For every $\omega \in \rH^0(A, \Omega^1_A)\setminus \{0\}$ the pull-back $\nu^\ast \omega$ has finite zero locus.\smallskip
\end{itemize}
\item If $\cN_{\tilde{X}/ A}$ is not ample, then there is a curve $C \subset X$ lying in a smaller abelian subvariety of $A$. In particular, if $A$ is simple, then $\cN_{\tilde{X}/ A}$ is ample.
\end{enumerate} 
\end{lemma}

\begin{proof} 
We may assume $k=\bbC$.  For claim (1) we begin by showing that the characteristic variety 
\[ \Characteristic(\delta_X) \;:=\; \Supp(\CC(\delta_X)) \;\subset\; T^\vee A \]
is irreducible. Note that $\delta_X = \nu_*(\delta_{\tilde{X}})$ since the normalization $\nu\colon \tilde{X} \to X$ is a finite birational morphism. So we may apply Kashiwara's estimate for the characteristic variety of direct images~\cite[th.~4.2(b)]{KashiwaraBFunctionsAndHolonomic}: the morphism $f\colon \tilde{X} \to A$ induces a correspondence
\[
 T^\vee \tilde{X} \;\stackrel{\rho}{\longleftarrow} \; \tilde{X} \times_X T^\vee A \;\stackrel{\varpi}{\longrightarrow}\; T^\vee A
\]
where $\varpi=f\times \id$ and $\rho = df$, and with this notation
\[
 \Characteristic(\delta_X) \;=\; \Characteristic(\nu_*(\delta_{\tilde{X}})) \;\subset\; \varpi(\rho^{-1}(\Characteristic(\delta_{\tilde{X}})).
\]
Here $\Characteristic(\delta_{\tilde{X}})$ is the zero section $T^\vee \tilde{X}$. Since $f$ is unramified, the codifferential $\rho$ is smooth equidimensional with integral fibers of dimension $r=g-d$ where $d=\dim X$, so $\rho^{-1}(\Characteristic(\delta_{\tilde{X}}))$ is integral of dimension $g$. Since $\varpi$ is a finite morphism, it follows that $\varpi(\rho^{-1}(\Characteristic(\delta_{\tilde{X}}))$ is integral of dimension $g$. So the inclusion in Kashiwara's estimate is an equality since $\Char(\delta_X)$ is of pure dimension $g$. In particular, the characteristic cycle $\cc(\delta_X)$ coincides with the conormal variety $\PLambda_X$, which is in turn the projectivization
\[ \bbP(\cC_{\tilde{X}/A}) \; \subset \; A\times \bbP_A\]
of the conormal bundle $\cC_{\tilde{X}/A}$ of the unramified morphism $\tilde{X} \to A$. Hence the projection $\PLambda_X \to X$ is the composite morphism
\[ \bbP(\cC_{\tilde{X}/A}) \; \too \; \tilde{X} \; \too \; X,\]
so it has constant fiber dimension. This proves claim (1). Claim (2) then follows from the Kashiwara index formula together with the identity 
\[ \chi(\delta_X) \;=\; \chi(\nu_*(\delta_{\tilde{X}})) \;=\; \chi(\delta_{\tilde{X}}) \;=\; (-1)^d \chi_\top(\tilde{X}) \]
where the last equality uses that $\tilde{X}$ is smooth. The claims (3) and (4) are proven in \cite[prop. 6.3.10]{LazarsfeldPositivityII} in the special case where $f$ is a closed embedding, but the proof goes through verbatim when $f$ is only unramified.
\end{proof}

\subsection{Lower bounds for the Euler characteristic} Recall that the group~$G_{X,\omega}$ comes with a natural faithful irreducible representation $\omega(\delta_X)$ whose dimension is given by $\dim \omega(\delta_X)=\chi(\delta_X)=|\chi_\top(\tilde{X})|$. In~\cref{Subsec:ConclusionSpin} we will use the following lower bound on this topological Euler characteristic:

\begin{lemma} \label{Lem:EulerCharacteristicBound}
Let $X \subset A$ be an integral subvariety of dimension $d<g$ that has smooth unramified normalization with ample normal bundle. Then
\[ |\chi_\top(\tilde{X})| \;\ge\; g.
\]
\end{lemma} 

\begin{proof}
We may assume that $k = \bbC$. The claim then follows from~\cite[th.~4]{BeltramettiSchneiderSommese} by identifying $\chi_\top(\tilde{X})$ with the top Segre class of the normal bundle,  since the normal bundle is ample, globally generated, and $\rH^1(X, \bbC)\neq 0$.
\end{proof} 

The above bound is clearly not sharp, as one already sees for curves. For surfaces we have the following bound:

\begin{lemma} \label{Lem:EulerCharacteristicBoundSurface}
Let $S$ be a smooth minimal surface of general type and $q=h^1(S, \cO_S)$ its irregularity. Then
\[ \chi_\top(S) \;\ge\; 3q-9. \]
\end{lemma} 

\begin{proof} 
Since $S$ is of general type, the Chern classes $c_i=c_i(T_S)$ satisfy $c_1^2\le 3c_2$ by the Bogomolov-Miyaoka-Yau inequality. By Noether's formula this is equivalent to~$3 \chi \le c_2$ where $\chi=\chi(S, \cO_S)=1-q+p$ for $p = h^2(X, \cO_X)$. As~$S$ is minimal, we have $p \ge 2q - 4$ by~\cite[appendix]{DebarreNoether}. This is equivalent to $\chi \ge q - 3$. Combining these inequalities we obtain $c_2 \ge 3(q-3)$. 
\end{proof}

\section{Simplicity}
\label{sec:Simplicity}

In this section we prove a criterion for the simplicity of Tannaka groups that applies to a large class of singular subvarieties of abelian varieties, including all subvarieties that receive a finite dominant morphism from a smooth variety.

\subsection{Main result} 

We say that an integral variety $X$ has \emph{uniruled modifications} if for every proper birational morphism 
\[ f\colon \quad X'\;\too\; X \]
from a normal variety $X'$ all irreducible components of the exceptional locus of~$f$ are uniruled over $X$. When $X$ is normal, this is the definition in~\cite[VI.1.6]{KollarRationalCurves}. The class of varieties that have uniruled modifications is quite large: in particular, by \cite[chapt.~VI, th.~1.2 and (1.6.2.3)]{KollarRationalCurves} it includes all varieties which receive a finite dominant morphism from a smooth variety. More generally, Hacon and McKernan have shown that any normal variety with divisorially log terminal singularities has uniruled modifications~\cite[cor.~1.6]{HaconMcKernan07}. The main goal of this section is the following result:

\begin{theorem} \label{Thm:SimpleTannakaGroups}
Let $X\subset A$ be an integral nondivisible subvariety with uniruled modifications and dimension $d < g -1$. If every component of $\cc(\delta_X)$ has finite Gauss map, then the following are equivalent:
\begin{enumerate} 
\item $G_{X,\omega}^\ast$ is not simple.\smallskip
\item $X=X_1+X_2$ for positive dimensional  integral subvarieties $X_1, X_2\subset A$ such that the sum morphism induces an isomorphism
\[
 \tilde{X}_1 \times \tilde{X}_2
 \;\stackrel{\sim}{\longrightarrow}\; \tilde{X}.
\]
\end{enumerate}
Moreover, all subvarieties $X_1, X_2\subset A$ with the property (2) have finite Gauss map.
\end{theorem} 

Before we come to the proof, which will take up the rest of this section, let us note that the theorem in particular applies to subvarieties $X\subset A$ which have smooth unramified normalization with ample normal bundle: for such subvarieties the characteristic cycle $\cc(\delta_X)$ is integral with finite Gauss map by~\cref{Lem:IntegralCC}. In such a situation also the summands $X_1, X_2\subset A$ have smooth unramified normalization with ample normal bundle:

\begin{corollary} \label{Cor:Product}
Let $X\subset A$ be an integral nondivisible subvariety  which has smooth unramified normalization with ample normal bundle and dimension $d<g-1$. For any decomposition $X=X_1 + X_2$ with integral subvarieties $X_i \subset A$ such that the sum map induces an isomorphism
\[
 \tilde{X}_1 \times \tilde{X}_2 \;\stackrel{\sim}{\longrightarrow}\; \tilde{X},
\]
both $X_i\subset A$ have smooth unramified normalization with ample normal bundle. 
\end{corollary}

Indeed, if $X$ has smooth unramified normalization, then so do both $X_i$. If $\tilde{X}\to A$ moreover has ample normal bundle, then $X$ has finite Gauss map by~\cref{Lem:IntegralCC}. In this case the last sentence in~\cref{Thm:SimpleTannakaGroups} shows that the Gauss maps of both $X_i$ are finite, hence $\tilde{X}_i \to A$ has ample normal bundle as claimed.

\subsection{A necessary criterion for simplicity} In the proof of~\cref{Thm:SimpleTannakaGroups}, we will apply the following criterion to show the implication (2) $\Rightarrow$ (1) and the finiteness of the Gauss map for the summands:

\begin{lemma} \label{Lem:ProductDecomposition}
Let $X\subset A$ be an integral subvariety such that every component of~$\cc(\delta_X)$ has finite Gauss map. Suppose there are two subvarieties $X_1, X_2\subset A$ of positive dimension such that $\delta_X \simeq \delta_{X_1} * \delta_{X_2}$. Then\smallskip
\begin{enumerate} 
\item the group $G_{X,\omega}^\ast$ is not simple, and \smallskip
\item the subvarieties $X_1, X_2\subset A$ have finite Gauss map.
\end{enumerate}
\end{lemma}

\begin{proof}
Replacing all the subvarieties by their image under the isogeny $[n]\colon A\to A$ for some $n\in \bbN$, we may assume that $X$ is nondivisible; moreover, replacing all the subvarieties by suitable translates we may assume that $H:=G_\omega(\delta_{X_1}\oplus \delta_{X_2})$ is a semisimple group~\cite[lemma~5.3.1]{KraemerMicrolocalII}. Since by assumption every component of~$\cc(\delta_X)$ has finite Gauss map, it then follows by~\cite[lemma~2.3]{KraemerThetaSummands} that the natural morphism 
\[
 H \;\twoheadrightarrow \; G_{X,\omega}
\] 
is an isogeny. Hence if $G_{X,\omega}^\ast$ were a simple algebraic group, then $H^*$ would be as well. But in 
\[
 \omega(\delta_X) \;\simeq\; \omega(\delta_{X_1})\otimes \omega(\delta_{X_2})
\]
the left hand side is a nontrivial irreducible representation of $H$ while the right hand side is a tensor product of \emph{two} nontrivial irreducible representations. This contradicts the uniqueness of the number of irreducible factors in a tensor product of representations of simple algebraic groups~\cite{Rajan}, thus proving (1). Claim (2) follows from the fact that 
\[
 [n]_* \cc(\delta_{X_i}) \;\in\; \bigl\langle \cc(\delta_X) \bigr\rangle
\]
for some $n\in \bbN$ by~\cite[prop.~2.5, lemma 2.6 and ex.~2.3]{KraemerThetaSummands}.
\end{proof}

\subsection{Proof of~\cref{Thm:SimpleTannakaGroups}}
(2) $\Rightarrow$ (1): If $X = X_1 + X_2$ where $X_1, X_2\subset A$ are integral subvarieties of positive dimension such that the sum morphism induces an isomorphism $\tilde{X}_1 \times \tilde{X}_2 \to \tilde{X}$,
then the sum morphism $X_1 \times X_2 \to X$ is finite and birational. It follows that
\[
 \delta_X \;\simeq \; \delta_{X_1} * \delta_{X_2}
\]
and~\cref{Lem:ProductDecomposition} shows that in this situation the Tannaka group $G_{X,\omega}^\ast$ is not simple and the subvarieties $X_1, X_2 \subset A$ again have finite Gauss map.

\medskip

(1) $\Rightarrow$ (2): If $G_{X,\omega}^\ast$ is not simple, then as in the proof of~\cite[prop.~6.10]{JKLM} there exist clean effective cycles $\PLambda_1, \PLambda_2 \in \langle \cc(\delta_X)\rangle$ of Gauss degree $\deg(\PLambda_i)>1$ and an integer $n\ge 1$ such that
\[
 \PLambda_1 \circ \PLambda_2 
 \;=\; \cc(\delta_Y)
 \quad \text{for} \quad Y \;=\; [n](X),
\]
where $\circ$ denotes the convolution product of clean cycles~\cite[def.~5.2]{JKLM}. The cycle on the right has the form $\cc(\delta_Y) = \PLambda_Y + R$ where $R$ is either zero or an effective clean cycle which is supported over a strict subvariety of $Y \subset A$. Since the convolution of effective clean cycles is effective, in the above product decomposition there must be components $\PLambda_{Y_i} \subset \Supp(\PLambda_i)$ over integral subvarieties $Y_i \subset A$ such that
\[
 \PLambda_{Y_1} \circ \PLambda_{Y_2} \;=\; \PLambda_{Y} + \cdots
\]
where $\cdots$ is a sum of components of $R$. By assumption all components of $\cc(\delta_X)$ have finite Gauss map, so for $i=1,2$ the Gauss map of $\PLambda_{Y_i} \in \langle \cc(\delta_X)\rangle$ is finite. Since $\dim X < g-1$, it follows that
\[
 \dim Y \;=\; \dim Y_1 + \dim Y_2
\]
by~\cite[cor.~5.7]{JKLM}. The product decomposition $\PLambda_{Y_1} \circ \PLambda_{Y_2} = \PLambda_{Y} + \cdots$ then shows that the sum morphism induces a birational morphism $Y_1\times Y_2 \to Y$. Note that $\dim Y_i > 0$ for both $i=1,2$: if not, then by irreducibility we could assume that $Y_1 = \{p\}$ is a point, in which case $Y_2 = Y - p$ by the product decomposition. But since 
\[
 \PLambda_1 \circ \PLambda_{Y_2} \;\subset\; \PLambda_1 \circ \PLambda_2 \;=\; \PLambda_Y + R,
\]
this could happen only if $\PLambda_1 = \PLambda_{\{p\}}$ which is impossible since $\deg(\PLambda_1)>1$.

\medskip 

It remains to deduce from the above decomposition of $Y$ a similar decomposition of $X$. Since $X\subset A$ is nondivisible, the morphism $[n]\colon X \to Y$ is birational. Hence if we pick irreducible components $X_i\subset A$ of the preimage $[n]^{-1}(X_i)$ for $i=1,2$ such that $X=X_1+X_2$, then the sum morphism $X_1 \times X_2 \to X$ is birational. By passing to normalizations we obtain a birational morphism
\[
 \tilde{\sigma}\colon \quad \tilde{X}_1 \times \tilde{X}_2 \;\longrightarrow\; \tilde{X}.
\]
Since $\tilde{X}$ has uniruled modifications, any positive-dimensional fiber of $\tilde{\sigma}$ would have to be uniruled. But $\tilde{X}_1 \times \tilde{X}_2$ admits a finite morphism to an abelian variety $A\times A$ and hence it cannot contain any uniruled subvariety of positive dimension. Thus the morphism $\tilde{\sigma}$ is finite, and the claim follows since any finite birational morphism between normal varieties is an isomorphism.   \qed

\section{Wedge powers} \label{sec:WedgePowers}

We now focus on subvarieties that have smooth unramified normalization with ample normal bundle. In this section we show that for such subvarieties of high enough codimension, wedge powers of the standard representation of $\SL_n(\bbF)$ occur only if the normalization is a symmetric power of a curve.

\subsection{Main result}
Let $r\ge 1$ be an integer. We say that a subvariety $X\subset A$ is an~\emph{$r$-th wedge power} if there exists an integer $n\ge r$ such that the Tannaka group is
\[
 G_{X, \omega}^\ast 
 \simeq \Alt^r(\SL_n(\bbF))
 \quad \textnormal{with the standard action on} \quad
 \omega(\delta_X)
 \simeq \Alt^r(\bbF^n).
\]
If $r$ is given, then $n$ is determined by the Euler characteristic of the intersection cohomology since
\[
\chi(\delta_X) \;=\; \dim_\bbF(\omega(\delta_X)) \;=\; \tbinom{n}{r}. \medskip
\]
By duality it suffices to discuss wedge powers for $r\le n/2$. In this case we have:

\begin{theorem} \label{Thm:SmallWedgePowersAreSumsOfCurves}
Let $X\subset A$ be an integral nondivisible subvariety which has smooth unramified normalization with ample normal bundle and dimension $d<(g-1)/2$. Suppose that~$\chi(\delta_X)=\binom{n}{r}$ with $1 < r \le n/2$. Then the following are equivalent: \medskip
\begin{enumerate} 
\item $X\subset A$ is an $r$-th wedge power.\smallskip
\item $r = d$ and there is an irreducible curve $C\subset A$ such that\smallskip 
\begin{itemize}
\item $X=C+\cdots + C \subset A$ is the sum of $d$ copies of $C$, and\smallskip 
\item the sum map induces on normalizations an isomorphism $\Sym^d \tilde{C} \to \tilde{X}$.\medskip
\end{itemize}
\end{enumerate} 
Any such $C\subset A$ has smooth unramified normalization with ample normal bundle.
\end{theorem}

The proof will proceed in three steps: in~\cref{subsec:conormal-of-wedge} we discuss the general structure of conormal varieties of wedge powers, then in~\cref{subsec:sum-of-wedge} we show the finiteness of the resulting sum morphism, and in~\cref{subsec:conclusion-of-wedge} we conclude by some general arguments about symmetric powers of varieties.

\subsection{Conormal varieties of wedge powers} \label{subsec:conormal-of-wedge}

Let $Z\subset A$ be a subvariety with dominant Gauss map. Then there exists an open dense subset $U\subset \bbP_A$ over which the Gauss map $\gamma_Z$ is finite and flat. For an integer $r\ge 1$ we consider
 the Zariski closure
\[
 \PLambda_{Z}^{[r]} \;:=\; 
 \overline{\PLambda^{\times r}_{Z\mid U} \smallsetminus \Delta_r } 
 \;\subset\; A^r \times \bbP_A
\]
and put 
\[
 \Alt^r \PLambda_Z \;:=\; \tfrac{1}{r!} \, \sigma_*(\PLambda_{Z}^{[r]})
 \;\in\; \cL(A).
\]
for the sum morphism $\sigma\colon A^r \times \bbP_A \to A\times \bbP_A, (z_1, \dots, z_r, \xi) \mapsto (z_1+\cdots + z_r, \xi)$.

\begin{proposition} \label{prop:cc-for-wedge-power}
Let $X\subset A$ be a nondivisible subvariety with $\dim X > 0$ and with integral characteristic cycle. If $\omega(\delta_X)$ is an $r$-th wedge power for some $r\ge 1$, then there exists an integer $e\ge 1$ and an integral subvariety $Z\subset A$ with $\PLambda_Z \in \langle \PLambda_X \rangle$ such that
\[
 \Alt^r \PLambda_Z \;=\; \PLambda_{[e](X)}.
\]
In this case we have
\[ \deg \PLambda_X = \tbinom{n}{r}
\quad \text{for} \quad n = \deg \PLambda_Z, \medskip \]
and if the Gauss map $\gamma_X\colon \PLambda_X \to \bbP_A$ is a finite morphism, then so is $\gamma_Z\colon \PLambda_Z \to \bbP_A$.
\end{proposition} 

\begin{proof}
In~\cite[th.~7.4]{JKLM} this was stated only when $X\subset A$ is a smooth integral subvariety, but the same proof applies verbatim also under the weaker assumption that $X\subset A$ is a subvariety with integral characteristic cycle.
\end{proof} 

\subsection{The sum morphism} \label{subsec:sum-of-wedge}
The next step in the proof of~\cref{Thm:SmallWedgePowersAreSumsOfCurves} is to use the dimension estimate for the image
\[
 \Alt^r Z \;:=\; 
 \im 
 \Bigl(
 \pr_A\colon \Supp(\Alt^r \PLambda_Z) \to A
 \Bigr)
 \;\subset\;
 A
\]
of the support of the cycle constructed above: for any integral subvariety $Z\subset A$ such that the Gauss map $\gamma_Z$ is a finite morphism and
$\dim \Alt^r Z < (\dim A - 1)/2$
for some $r\le \deg(\gamma_Z)/2$, we have
\[
 r \dim Z \;<\; \dim A
\]
by~\cite[th.~7.5]{JKLM}. We can therefore apply the following result to $Y=[e](X)$:

\begin{proposition} \label{prop:sum-for-wedge-power}
Let $Y, Z\subset A$ be integral subvarieties with finite Gauss map such that
\[
 \Alt^r \PLambda_Z \;=\; \PLambda_Y
\]
for some $r\ge 1$ with $r\dim Z < \dim A$. Then the following properties hold:
\begin{enumerate} 
\item The sum morphism $\sigma\colon Z^r \to A$ has image $Y\subset A$ and
induces a birational morphism
\[
 \tau\colon \quad \Sym^r Z \;=\; Z^r/\frS_r \;\longrightarrow\; Y
\]
\item If $\pr_Y\colon \PLambda_Y \to Y$ has constant fiber dimension, then $\tau$ is finite.
\end{enumerate}
\end{proposition}

\begin{proof} 
The finiteness of the Gauss map $\gamma_Z$ implies that the subvariety $Z\subset A$ is nondegenerate~\cite[th.~2.8]{JKLM}. As in the proof of~\cite[prop.~7.12]{JKLM}, the assumption $r\dim Z < \dim A$ thus implies that $Y\subset A$ is a sum of $r$ copies of $Z$ and that the sum morphism $\tau\colon \Sym^r Z \to Y$ is birational, thus proving (1). Claim~(2) amounts to the statement that if the projection $\pr_Y$ has contant fiber dimension, then $\sigma$ is a finite morphism. This follows as in loc.~cit.~from the commutative diagram
\[
\begin{tikzcd}[column sep=40pt]
 Z^r  \ar[d, swap, "\sigma"] &  \PLambda_Z^{[r]}  \ar[d, "\tilde{\sigma}"]  \ar[l, swap, "\pr_{Z,r}"]  \ar[r,"\gamma_{Z,r}"] &  \bbP_A  \ar[d, equals] \\
Y &  \PLambda_Y \ar[l, "\pr_Y"] \ar[r, swap, "\gamma_Y"]  & \bbP_A
\end{tikzcd} 
\]
by the semicontinuity of the fiber dimension for proper morphisms, using that $\tilde{\sigma}$ is a finite morphism since the Gauss maps $\gamma_Z$ and $\gamma_{Z,r}$ are finite.
\end{proof} 

Note that $\PLambda_Y \to Y$ has constant fiber dimension when $Y = [e](X)$ for some integer $e \ge 1$ and for some nondivisible integral subvariety $X \subset A$ with smooth unramified normalization; see \cref{Lem:IntegralCC} (1).

\subsection{Conclusion of the proof of~\cref{Thm:SmallWedgePowersAreSumsOfCurves}}
\label{subsec:conclusion-of-wedge}

Let $X\subset A$ be an integral nondivisible subvariety of dimension $d<(g-1)/2$ which has smooth unramified normalization with ample normal bundle, and suppose that $\chi(\delta_X)=\tbinom{n}{r}$ with $1<r\le n/2$. 

\medskip 

If condition (1) in the theorem holds, then by~\cref{prop:cc-for-wedge-power,prop:sum-for-wedge-power} we can find an integral subvariety $Z\subset A$ and an integer $e\ge 1$ with the property that the sum morphism induces a finite birational morphism 
$ \Sym^r Z \rightarrow Y =[e](X)$.
Taking preimages under $[e]\colon A\to A$, the same argument as in~\cite[proof of th.~7.6]{JKLM} goes through without changes and shows that the sum morphism induces a finite birational morphism
\[
 \Sym^r C \;\too \;X,
\]
where $C\subset A$ is a suitable translate of an irreducible component of $[e]^{-1}(Z)$. So we get an isomorphism
\[
 \Sym^r \tilde{C} \;\stackrel{\sim}{\longrightarrow}\; \tilde{X}
\]
between the respective normalizations. Since by assumption the normalization $\tilde{X}$ is smooth, it follows that $C$ must be a curve. Hence (2) holds. Moreover, since we assumed the morphism $\tilde{X}\to X$ to be unramified, we know that $\Sym^r \tilde{C} \to \Sym^r C$ is unramified. By looking at points over the diagonal one sees that this can happen only if $\tilde{C} \to C$ is unramified. 

\medskip 

Conversely, suppose that condition (2) of the theorem holds for some integral nondegenerate curve $C\subset A$ and some integer $r>1$. We have a commutative diagram
\[ 
\begin{tikzcd} 
\Sym^r \tilde{C} \ar[r, "\tilde{\tau}_r"] \ar[d, swap, "\mu_r"] & \tilde{X} \ar[d, "\nu_r"]
\\
\Sym^r C \ar[r, "\sigma_r"] & X
\end{tikzcd}
\]
where $\tilde{\tau}_r$ is an isomorphism and $\mu_r$ and $\nu_r$ are finite birational since they are normalization morphisms. To reduce to the case $r=2$, we fix $r-2$ general points on $\tilde{C}$ and embed $\Sym^2 \tilde{C}$ into $\Sym^r \tilde{C}$ by adding those points. Then~$\tilde{\tau}_r$ restricts to a morphism
\[
 \Sym^2 \tilde{C} \;\longrightarrow\; \tilde{X}
\]
which is still an isomorphism onto its image. Denote this image by $\tilde{W}_2\subset \tilde{X}$ and put $W_2 = \nu(\tilde{W}_2)$. We get a commutative diagram
\[
\begin{tikzcd}
\Sym^2 \tilde{C} \ar[r, "\tilde{\tau}_2"] \ar[d, swap, "\mu_2"] & \tilde{W}_2 \ar[d, "\nu_2"]
\\
\Sym^2 C \ar[r, "\tau_2"] & W_2
\end{tikzcd}
\]
where $\tilde{\tau}_2$ is still an isomorphism and where $\mu_2$ and $\nu_2$ are still finite birational because we fixed $r-2$ \emph{general} points of $\tilde{C}$. Hence it follows that $\tau_2$ is a finite birational morphism. By precomposing it with the double cover $C^2 \to \Sym^2 C$ we see that the sum morphism
$\sigma_2 \colon C^2 \too W_2$
is finite of degree two, so the convolution square
\[
 \delta_C * \delta_C \;=\; \sigma_{2*}(\delta_{C^2})
\]
has only two perverse direct summands. It then follows from Larsen's alternative that $G_{X,\omega}^\ast \simeq \SL_n(\bbF)$ where $n=\chi(\delta_C)$, see~\cite[lemma~3.7]{JKLM}. Moreover, by construction
\[
 \omega(\delta_X) \;\simeq\; \omega(\tau_{r*}(\delta_C)) \;\simeq\; \Alt^r(\omega(\delta_C)),
\]
hence condition (1) in the theorem holds. 
\qed

\section{Spin representations} \label{sec:Spin}

We now exclude spin representations for subvarieties of high codimension that have smooth unramified normalization with ample normal bundle.

\subsection{How to rule out the spin case} 
For $n \ge 1$, we say that a subvariety $X\subset A$ is 
\begin{itemize} 
\item {\em of spin type $B_n$} if its Tannaka group is the spin group $G_{X, \omega}^\ast \simeq \Spin_{2n+1}(\bbF)$,\smallskip
\item {\em of spin type $D_n$} if its Tannaka group is a half-spin group $G_{X, \omega}^\ast \simeq \Spin_{2n}^\pm(\bbF)$.\smallskip
\end{itemize} 
In this definition we do not make any assumption on the representation $V_X=\omega(\delta_X)$ of the Tannaka group, since it will be automatically determined in the case relevant to us:
If a nondivisible subvariety $X\subset A$ has smooth unramified normalization, then by~\cref{Lem:IntegralCC} it has integral characteristic cycle, so the group $G_{X,\omega}^\ast$ acts on~$V_X$ via a minuscule representation \cite[cor.~1.10]{KraemerMicrolocalI}~\cite[cor.~5.15]{JKLM}. A look at the list of faithful minuscule representations then shows:\smallskip
\begin{itemize} 
\item If $X$ is of spin type $B_n$, then $V_X$ is the spin representation.\smallskip
\item If $X$ is of spin type $D_n$, then $V_X$ is the associated half-spin representations.\smallskip
\end{itemize}
We exclude both cases in high enough codimension:

\begin{theorem}\label{Thm:SmallSpinDoNotExist} Let $X\subset A$ be an integral nondivisible subvariety which has smooth unramified normalization with ample normal bundle and dimension $d<(g-1)/2$.\smallskip
\begin{enumerate}
\item The subvariety $X\subset A$ is not of spin type $B_n$ for any $n\ge 1$. \smallskip
\item If $X$ is of spin type $D_n$ for some $n$, then $d \ge (g - 1)/4$ and $n = 2m$ for some integer $m \in \{3,\dots,d\}$ which has the same parity as $d$.
\end{enumerate}
\end{theorem}

The proof of this follows a similar pattern as the argument for wedge powers and will take up the rest of this section.

\subsection{Conormal varieties in the spin case} 

Let $Z$ be a symmetric reduced clean effective cycle on $A$ with dominant Gauss map. Then there exists an open dense subset $U\subset \bbP_A$ over which the Gauss map $\gamma_Z$ is finite \'etale. For an integer $r\ge 1$ we consider the Zariski closure
\[
\PLambda_{Z,\sym}^{[r]} \;:=\; 
\overline{\PLambda_{Z\vert U}^{\times r} \smallsetminus (\Delta_r \cup \Delta_r^-)} \;\subset\; \PLambda_Z^{\times r}
\] 
and put
\[
 \Alt^{r}_\sym \PLambda_Z  \;:=\; \sigma_{*}(\PLambda_{Z, \sym}^{[r]})
\]
for the sum morphism $\sigma \colon A^r \times \bbP_A \to A\times \bbP_A$. We say the Gauss map $\gamma_Z\colon \PLambda_Z \to \bbP_A$ has {\em even monodromy} if its degree is an even integer $\deg(\gamma_Z)=2n$ and the finite \'etale cover $\gamma_{Z\vert U}\colon \PLambda_{Z\vert U} \to U$ has its monodromy group contained in the index two subgroup
\[ H \;:=\; \{\pm 1\}^n_+\rtimes \frS_n 
\;\subset\; \{\pm 1\}^n \rtimes \frS_n \]
where $\{\pm 1\}^n_+ := \{ (\epsilon_1, \dots, \epsilon_n) \in (\pm 1)^n \mid \epsilon_1 \cdots \epsilon_n = 1\}$. Looking at the orbits of this index two subgroup on a general fiber of the Gauss map, one obtains a decomposition
\[
 \PLambda_{Z, \sym}^{[n]} \;=\; \PLambda_{Z, \sym, +}^{[n]} +  \PLambda_{Z, \sym, -}^{[n]}
\]
where $H$ acts transitively on the fibers of each of the two summands $\PLambda_{Z, \sym, \pm}^{[n]}$ on the right hand side. Note that there is no intrinsic way to distinguish the two summands from each other. For $r=n$ we put
\[
 \Alt^n_{\sym, \pm} \PLambda_Z \;:=\; \sigma_{*}(\PLambda_{Z, \sym, \pm}^{[n]}).
\] 
These appear naturally as characteristic cycles in the spin case:

\begin{proposition} \label{Prop:ConormalToSpin}
Let $X\subset A$ be a nondivisible subvariety with dominant Gauss map and integral characteristic cycle. If $X$ is of spin type $B_n$ or $D_n$ for some $n$, then there exists $a\in A(k)$, a reduced symmetric effictive cycle $Z$ on $A$ of Gauss degree $\deg \PLambda_Z = 2n$ and an integer $e\ge 1$ with the following properties:\smallskip
\begin{enumerate} 
\item If $X$ is of spin type $B_n$, then $\Alt^n_\bbS \PLambda_Z = \PLambda_{[e](X+a)}$.\smallskip
\item If $X$ is of spin type $D_n$, then the Gauss map $\gamma_Z$ has even monodromy and we have
\[
 \Alt^n_{\bbS, \epsilon} \PLambda_Z \;=\; \PLambda_{[e](X+a)}
 \quad \text{for some} \quad \epsilon \;\in\; \{ +, -\}.
\]
\end{enumerate}
In both cases, if $\gamma_X\colon \PLambda_X \to \bbP_A$ is a finite morphism, then so is $\gamma_Z\colon \PLambda_Z \to \bbP_A$.
\end{proposition}

\begin{proof}
In~\cite[th.~8.5]{JKLM} this was stated only when $X\subset A$ is a smooth integral subvariety, but the same proof applies verbatim also under the weaker assumption that $X\subset A$ is a subvariety with integral characteristic cycle.
\end{proof} 

\subsection{The sum morphism} 

The next step in the proof of~\cref{Thm:SmallSpinDoNotExist} is to consider the image
\[
 \Alt^n_{\bbS, \epsilon} Z \;:=\; 
 \pr_A(\Supp(\Alt^n_{\bbS, \epsilon})) 
 \;\subset\; A
 \quad \text{for} \quad \epsilon \;\in\; \{+, -, \varnothing\}.
\]
For any reduced symmetric effective cycle $Z$ on $A$ whose Gauss map is a finite morphism of even degree $\deg(\gamma_Z)=2n$, we know from~\cite[th.~8.6]{JKLM}:\smallskip
\begin{enumerate} 
\item If $\dim \Alt^n_\bbS Z < (g-1)/2$, then $n\dim Z < g-1$.\smallskip
\item If $\gamma_Z$ has even monodromy and there is $\epsilon \in \{+,-\}$ for which $r:=\dim \Alt^n_{\bbS, \epsilon} Z$ satisfies
\[
 r \;<\;
 \begin{cases}
 (g-1)/4 & \text{if $n=2m$ is even and $m\le r+1$}, \\
 (g-1)/2 & \text{otherwise},
 \end{cases}
\]
then $n\dim Z < g-1$.
\end{enumerate}  
In what follows, we will use the uniform notation $\Alt^n_{\bbS, \epsilon} \PLambda_Z$ for $\epsilon \in \{+,-,\varnothing\}$ to include all the above cases. For $\epsilon \in \{ \pm 1\}$ this notation includes the tacit assumption that the Gauss map $\gamma_Z$ has even monodromy. With this convention we have the following dimension estimate in the opposite direction to (1) and (2) above:

\begin{proposition} \label{Prop:SpinLowerBound} 
Let $X\subset A$ be a nondivisible subvariety with finite Gauss map and integral characteristic cycle such that the projection $\pr_X\colon \PLambda_X \to X$ has constant fiber dimension. Suppose that there is a reduced symmetric effective cycle $Z$ on $A$ with finite Gauss map of degree $2n$ such that
\[ \Alt^n_{\bbS, \epsilon} \PLambda_Z = \PLambda_{[e](X)} \]
for some integers $n\ge 2$, $e\ge 1$ and some $\epsilon \in \{+,-,\varnothing\}$. Then $n\dim Z \ge g$.
\end{proposition}

\begin{proof} 
In~\cite[th.~8.7]{JKLM} this was stated only when $X$ is smooth, but the proof goes through without change in the more general setting considered here.
\end{proof}

\subsection{Conclusion of the proof of~\cref{Thm:SmallSpinDoNotExist}} \label{Subsec:ConclusionSpin}

Let $X\subset A$ be a nondivisible subvariety of dimension $d$ with smooth unramified normalization and ample normal bundle. Suppose $X$ is of spin type $B_n$ or $D_n$ for some integer $n\ge 1$. Replacing $X$ by some translate, we obtain from~\cref{Prop:ConormalToSpin} an integer $e\ge 1$ and a reduced symmetric effective cycle $Z$ on $A$ with finite Gauss map of degree $\deg(\gamma_Z)=2n$ such that
\[ \Alt^{n}_{\sym, \epsilon} \PLambda_Z \;=\; \PLambda_{[e](X)}
\quad \textnormal{for some} \quad \epsilon \;\in\; \{+, -, \varnothing\}. \]
Since $X$ has smooth unramified normalization, the projection $\pr_X\colon \PLambda_X \to X$ has constant fiber dimension. Moreover, the Gauss map $\gamma_Z$ is finite, so~\cref{Prop:SpinLowerBound} shows $n\dim Z \ge g$. But
\[ \dim \Alt^{n}_{\sym, \epsilon} Z \;=\; \dim [e](X) \;=\; \dim X \;=\; d.	\]
Hence by the above items (1) and (2) with $r=d$ together with~\cite[lemma~8.8]{JKLM} we have
\[ d \;\ge\; (g-1)/2 \]
unless $X$ is of spin type $D_{2m}$ for some integer $m\le d+1$, in which case 
$ d \ge (g-1)/4$.
In this case $m$ must have the same parity as $d$ because the half-spin representation is orthogonal for even $m$ and symplectic for odd $m$~\cite[rem.~8.2]{JKLM}. It only remains to exclude the case where $m=2$ and $d$ is even. But in this case we would have $\chi_\top(\tilde{X})=\chi(\delta_X)=8$, which is impossible by~\cref{Lem:EulerCharacteristicBound,,Lem:EulerCharacteristicBoundSurface}.
\qed

\section{Exceptional groups} \label{sec:Exceptional}

Finally we rule out the appearance of exceptional groups as Tannaka groups in the cases relevant to theorems \ref{IntroThmIntrinsicSurfaces} and \ref{IntroThmIntrinsicGeneral}.

\subsection{How to rule out exceptional groups}  The main goal of this section is the following result:

\begin{theorem} \label{Thm:NoExceptionals} Let $X\subset A$ be an integral subvariety which has smooth unramified normalization with ample normal bundle and dimension $d < (g-1)/2$. Then,
\[G_{X, \omega}^\ast  \; \not \iso \;  E_6, E_7.\]
\end{theorem}

This statement is sharp because for $(d, g) = (2, 5)$ the group $E_6$ occurs as the Tannaka group of the Fano surface~$X$ of lines on a smooth cubic threefold, embedded in $A = \Alb(X)$. Actually, one can show as in \cite[th. A]{KLM} that for $d<g/2$ this is the only example of subvarieties with exceptional Tannaka group. Since we do not need this for theorems \ref{IntroThmIntrinsicSurfaces} and \ref{IntroThmIntrinsicGeneral}, we here limit ourselves to the following positive statement:

\begin{proposition} \label{prop:E6Tannaka} 
Suppose $X\subset A$ is nondivisible of dimension $d<g/2$ and has smooth unramified normalization with ample normal bundle. If~$G_{X, \omega}^\ast \iso E_6$, then $X$ is surface with 
\[ \chi(\tilde{X},  \cO_{\tilde{X}}) \;=\; 6, \quad \; c_2(\tilde{X}) \;=\; 27 \quad \text{and} \quad g \;=\; 5. \]
\end{proposition}

The proof of this will be given in~\cref{sec:ProofOfE6Tannaka}.
To rule out the group $E_7$ we will use a simpler argument, based on the decomposition of the convolution square. Recall that
\[
 \delta_X * \delta_X \;=\; \Sym^2(\delta_X) \; \oplus \; \Alt^2(\delta_X)
\] 
where the symmetry constraint of the convolution product acts trivially  on $\Sym^2(\delta_X)$ and by $-\id$ on $\Alt^2(\delta_X)$. The representation theory of $E_7$ then implies:

\begin{proposition} \label{prop:E7}
Let $X\subset A$ be an integral nondivisible subvariety which has smooth unramified normalization with ample normal bundle. If~$G_{X,\omega}^\ast \iso E_7$, then $d=\dim X$ is odd, we have $X=a-X$ for some $a\in A(k)$, and
\[
 \Alt^2(\delta_X) \;\iso\; P \oplus \delta_a \oplus N
\]
for a simple perverse sheaf $P\in \Perv(A, \bbF)$ and a negligible complex $N\in \Dbc(A, \bbF)$.
\end{proposition}

The proof is the same as in~\cite[cor.~8.2]{JKLM}, since by~\cref{Lem:IntegralCC} (1) the cycle $\cc(\delta_X)$ is integral. To exclude the situation in the above proposition for subvarieties of dimension $d< g/2$ we then use as in loc.~cit.~a stronger geometric version of Larsen's alternative. It will be more natural to state the latter in a more general setup. For a subvariety $X\subset A$ of dimension $d$, we define
\[
T_+(\delta_X) \;:=\;
 \begin{cases} 
 \Sym^2(\delta_X) & \text{if $d$ is even}, \\
 \Alt^2(\delta_X) & \text{if $d$ is odd}.
 \end{cases} 
\]
If~$X = a - X$ for some $a \in A(k)$, then $\delta_a$ is a direct summand in $T_+(\delta_X)$. If $X$ is nondivisible, then $a$ is unique and we denote by $S_+(\delta_X)$ a complementary direct summand so that
\[ T_+(\delta_X) \;=\; S_+(\delta_X) \oplus \delta_a. \]
The skyscraper summand gives rise to a bilinear form $\theta \colon  V \otimes V \to \omega(\delta_a)$ which is symmetric if $d$ is even and alternating if $d$ is odd. Hence the derived subgroup~$G_{X,\omega}^\ast$ of the connected component of $G_{X, \omega}$ is contained in $\SO(V, \theta)$ is $d$ is even and in~$\Sp(V, \theta)$ if $d$ is odd. To have a uniform notation, in the case where~$X$ is not symmetric up to translation we set
\[ S_+(\delta_X) \; := \; T_+(\delta_X).\]
We then have the following version of Larsen's alternative:

\begin{proposition} \label{th:LarsenAlternative} Let $X\subset A$ be an integral  nondivisible subvariety which has smooth unramified normalization with ample normal bundle and dimension $d<g/2$. Then 
$ S_+(\delta_X)$ is a perverse sheaf without negligible direct summands. If this perverse sheaf is simple, then
\[
G_{X, \omega}^\ast \;=\; 
\begin{cases}
\SL(V) & \text{if $X$ is not symmetric up to translation}, \\
\SO(V, \theta) & \text{if $X$ is symmetric up to translation and $d$ is even}, \\
\Sp(V, \theta) & \text{if $X$ is symmetric up to translation and  $d$ is odd}.
\end{cases}
\]
\end{proposition}

The proof of this will be given in~\cref{sec:ProofOfLarsen}. Putting everything together, we can rule out exceptional Tannaka groups as follows:

\begin{proof}[{Proof of \cref{Thm:NoExceptionals}}]
If $X \subset A$ is the preimage of a subvariety $Y \subset B$ under an isogeny $A \to B$, then we have $G_{X,\omega}^\ast \simeq G_{Y, \omega}^\ast$
by~\cite[cor. 3.5]{JKLM}, and $Y$ still has  smooth unramified normalization with ample normal bundle. Hence in what follows we will assume that $X\subset A$ is nondivisible.

\medskip 

In this case, \cref{prop:E6Tannaka} rules out the Tannaka group~$E_6$ since $(d, g)  =  (2, 5)$ does not satisfy the condition $d < (g - 1)/2$ in \cref{Thm:NoExceptionals}. Likewise, for $d<g/2$ the Tannaka group $E_7$ is ruled out by combining~\cref{prop:E7,,th:LarsenAlternative}. 
\end{proof}

\subsection{Proof of \cref{prop:E6Tannaka}} \label{sec:ProofOfE6Tannaka}
By \cite[cor.~4.4]{JKLM} we may assume that~$k = \bbC$ and work analytically.  We first prove that $X$ is a surface with
\[ \chi(\tilde{X}, \cO_{\tilde{X}})=6 \quad \text{and} 
 \quad c_2(\tilde{X})= 27.
\]
As noted above, $G_{X,\omega}^\ast$ acts on~$V := \omega(\delta_X)$ via a minuscule representation.
The only minuscule representations of~$E_6$ are the irreducible representations of dimension~$27$, hence
\[ c_2(\tilde{X}) \; = \; \chi_{\top}(\tilde{X}) \; = \; 27.\]
For the remaining claims we use Hodge theory. Recall that up to isomorphism the group $G_{X, \omega}^\ast$ does not depend on our choice of the fiber functor $\omega$. In what follows we fix a torsion character $\chi \colon \pi_1(A, 0) \to \bbC^\times$ with $\rH^i(X, L_\chi)=0$ for all $i\neq d$ and take
\[ \omega\colon \quad \langle \delta_X \rangle \;\too\; \Vect(\bbC), \qquad P \; \longmapsto \; \rH^0(A, P \otimes L_\chi). \]
Let $\cL_\chi\in \Pic(A)$ be the line bundle defined by the torsion character $\chi$, and let $r$ be the order of this character. By~\cite[prop. B.1]{KM} the Hodge decomposition of the normalization of $[r]^{-1}(X)$ gives a decomposition of the vector space $V = \rH^d(\tilde{X}, L_\chi)$ as a direct sum
\[ V \; = \; 
\bigoplus_{p = 0}^d \rH^p(\tilde{X}) \qquad \textup{where} \qquad \rH^p(\tilde{X}) \; = \; \rH^{d - p}(\tilde{X}, \Omega^p_{\tilde{X}} \otimes \cL_\chi).\]
Moreover,  we have
\[ \dim \rH^p(\tilde{X}) \; = \; (-1)^{d - p}\chi(\tilde{X}, \Omega^p_{\tilde{X}}) \]
by \cite[prop. B.3]{KM}. The normal bundle of the unramified morphism $\tilde{X} \to A$ is ample, so
 the zero locus on $\tilde{X}$ of any nonzero differential form $\omega \in \rH^0(A, \Omega^1_A)$ is finite by \cref{Lem:IntegralCC} (3). So by \cite[th. A.1]{KM} we have 
\[ 
 (-1)^{d-p}\chi(\tilde{X}, \Omega^p_{\tilde{X}}) \;\ge\; 
 \begin{cases} 
  g - d +1& \text{for~$p \,\in\, \{0,d\}$},\\[0.1em]
   2 & \text{for~$p \,\in\, \{ 1, d - 1 \}$}, \\[0.1em]
 1 & \text{for~$2 \le p \le d-2$}.
 \end{cases} 
 \]
Let $\lambda \colon \Gm \to \GL(V)$ be the cocharacter such that $\lambda(z)$ acts by $z^{2p - d} \id$ on $\rH^p(\tilde{X})$. By~\cite[th. C]{KLM} this cocharacter takes values in the subgroup $G_{X, \omega}^\ast \subset \GL(V)$, hence
\[ d \; = \; 2 \quad \text{and} \quad h^0(\tilde{X}) \; = \; \chi(\tilde{X}, \cO_{\tilde{X}}) \; = \; 6\]
due to \cite[prop. 4.3]{KLM}. To apply this latter result, note that the cocharacter $\lambda$ has the properties (H1), (H2), (H3) in loc.~cit., see also \cite[prop. B.3]{KM}.\medskip

It remains to show $g = 5$, for which we will need to unveil the link with cubic threefolds. To begin with, Noether's formula and the Hirzebruch-Riemann-Roch theorem yield $c_1(\tilde{X})^2 = 45$ and $\chi(\tilde{X}, \Omega^1_{\tilde{X}}) = -15$. Consider the short exact sequence of vector bundles
\begin{equation} \label{eq:SESNormalE6}
 0 \too \cT \too \Lie A \otimes_k \cO_{\tilde{X}} \too \cN \too 0,
 \end{equation}
 where $\cT$ is the tangent bundle of the smooth variety $\tilde{X}$ and
  $\cN$ is the normal bundle of the unramified morphism $\tilde{X} \to A$. From this sequence we obtain
\[ c_1(\cN)^2 = 45 \quad \text{and} \quad c_2(\cN) = c_1(X)^2 - c_2(X) = 18.\]
The Gauss map~$\gamma\colon \tilde{X} \to \Gr_2(\Lie A)$, $x \mapsto \rT_x \tilde{X}$ is a finite morphism because $X$ is nondivisible \cite[prop. 3.1]{Deb95}. By interpreting~$c_2(\cN)$ as the top Segre class of~$\cT$, the positivity of~$c_2(\cN)$ implies that the morphism~$\pi \colon \bbP(\cT) \to \bbP(\Lie A)$ is generically finite onto its image. One easily checks that the generic degree of $\pi$ is bounded by
\[ \deg \pi \; \ge \; \deg D
\quad \text{for the difference morphism} \quad D\colon X\times X \;\too\; X-X, \]
see e.g.~\cite[lemma 7.3]{KLM} (the cited result is only for smooth subvarieties $Z \subset A$, but the same proof works for smooth varieties with a morphism $Z \to A$ birational onto its image). A key point is now that
\[ \deg D \; \ge \; 6\]
by the same argument as in \cite[prop. 4.6]{KLM} (which used smoothness only to ensure $\delta_X$ has integral characteristic cycle). By construction the degree $\deg(\pi)$ divides~$c_2(\cN) = 18$, thus
\[ \deg \pi \in \{ 6, 9, 18 \}. \]
Furthermore, the morphism $\tilde{X} \to \bbP(\Alt^2 \Lie A)$ given by $\gamma$ followed by the Pl\"ucker embedding of $\Gr_2(\Lie A)$ is the one defined by the canonical bundle $\cK_{\tilde{X}}$ of $\tilde{X}$. In particular $\deg \gamma$ must divide $c_1(\cK_{\tilde{X}}) = c_1(\tilde{X})^2 = 45$, hence
\begin{equation} \label{Eq:DegreeOfGaussMapE6}  \deg \gamma \in \{ 1,3, 9\}. \end{equation}
It follows that image $Y \subset \bbP(\Lie A)$ of~$\pi$ has dimension~$3$ and degree
\[ \deg Y = c_2(\cN) / \deg \pi \in \{1, 2,  3\}.\]
The subvariety~$X$ generates the abelian variety~$A$ since it is nondegenerate, so $Y$ is not contained in any hyperplane; hence~$\deg Y \neq 1$. The classical lower bound (see e.g.~\cite[prop. 0]{EisenbudHarrisMinimalDegree})
\[\deg Y \ge 1 + \codim Y\]
implies that~$Y$ has codimension~$1$ or~$2$ in $\bbP(\Lie A)$. Up to replacing $X$ by $\tilde{X}$, the proof of \cite[prop. 7.4]{KLM} goes through without changes and shows that the codimension $2$ case is not possible, hence $g=5$ as required. \qed

\subsection{Proof of \cref{th:LarsenAlternative}} \label{sec:ProofOfLarsen}
With notation as in \cref{th:LarsenAlternative}, let $T_-(\delta_X)$ be the direct summand complementary to $T_+(\delta_X)$ in the convolution square $\delta_X \ast \delta_X$ so that
\[ \delta_X*\delta_X \;=\; T_+(\delta_X) \oplus T_-(\delta_X). \]
We then have the following result:
 
\begin{lemma} \label{prop:larsen-alternative}
Suppose that $X\subset A$ is nondivisible of dimension $d<g/2$ and has smooth unramified normalization with ample normal bundle.\smallskip
\begin{enumerate}
\item The convolution $\delta_X * \delta_X$ is a perverse sheaf without negligible summands.\smallskip 
\item If $S_+(\delta_X)$ is simple, then \smallskip
\begin{enumerate} 
\item the sum morphism $f\colon \Sym^2 X \to X + X$ is birational, and\smallskip
\item $T_-(\delta_X)$ is a simple perverse sheaf with support $X+X$.
\end{enumerate} 
\end{enumerate}
\end{lemma}

\begin{proof} The proof is the same as in \cite[cor.~8.7]{KLM} except for the following two changes: in (1) the Gauss map $\PLambda_X \to \bbP_A$ is finite by \cref{Lem:IntegralCC}~(3). In (2) one has to replace $\Sym^2 X$ by $\Sym^2 \tilde{X}$. The latter is a rational homology manifold since~$\tilde{X}$ is smooth \cite[prop.~A.1(iii)]{BrionRationalSmoothness}; hence the intersection complex on $\Sym^2 \tilde{X}$ is the constant sheaf shifted in degree $-2d$ by~\cite[prop.~8.2.21]{Hotta}.
\end{proof}

\begin{proof}[{Proof of \cref{th:LarsenAlternative}}] Part (1) of \cref{prop:larsen-alternative} shows in particular that $S_+(\delta_X)$ is a perverse sheaf. Hence it only remains to show that if this perverse sheaf is simple, then
\[
G_{X, \omega}^\ast = 
\begin{cases}
\SL(V) & \text{if $X$ is not symmetric up to translation}, \\
\SO(V, \theta) & \text{if $X$ is symmetric up to translation and $d$ is even}, \\
\Sp(V, \theta) & \text{if $X$ is symmetric up to translation and  $d$ is odd}.
\end{cases}
\]
But this follows from part (2b) of \cref{prop:larsen-alternative} and \cite[lemma 3.7]{JKLM}.
\end{proof}

\small

\bibliography{./../biblio}

\bibliographystyle{amsalpha}

\end{document}